\documentclass[11pt]{article}
\usepackage[margin =1in]{geometry}
\usepackage{amssymb,amsthm,amsmath}
\usepackage{enumerate}
\usepackage{hyperref}
\usepackage{cite}
\usepackage[capitalize]{cleveref}
\usepackage{enumitem}
\usepackage{color}
\usepackage{cancel}
\usepackage{pgf}
\usepackage{svg}
\usepackage{pgfplots}
\usepackage{import}

\usepackage[utf8]{inputenc} % allow utf-8 input
\usepackage[T1]{fontenc}    % use 8-bit T1 fonts
\usepackage{hyperref}       % hyperlinks
\usepackage{url}            % simple URL typesetting
\usepackage{booktabs}       % professional-quality tables
\usepackage{amsfonts}       % blackboard math symbols
\usepackage{nicefrac}       % compact symbols for 1/2, etc.
\usepackage{microtype}
\usepackage{multirow}
\usepackage{xfrac}
\usepackage{todonotes}
\usepackage{titlesec}
\usepackage{caption}
\usepackage{subcaption}

\usepackage{multirow}

\usepackage{caption}
\usepackage{subcaption}
\usepackage{float}

\titleformat{\subsubsection}[runin]
{\normalfont\normalsize\bfseries}{\thesubsubsection}{1em}{}

\usepackage{graphicx}

\usepackage{appendix}
\numberwithin{equation}{section}
\newcommand{\inclu}[0] {\ar@{^{(}->}}

\newcommand{\argmin}{\operatornamewithlimits{argmin}}

\newtheorem{thm}{Theorem}[section]
\newtheorem{theorem}{Theorem}[section]

\newtheorem{proposition}[thm]{Proposition}

\newtheorem{lemma}[thm]{Lemma}

\crefname{claim}{claim}{claims}
\Crefname{claim}{Claim}{Claims}
\crefname{lem}{lemma}{lemmas}
\Crefname{lem}{Lemma}{Lemmas}
\crefname{algorithm}{algorithm}{algorithms}
\Crefname{algorithm}{Algorithm}{Algorithms}

\newtheorem{remark}{Remark}

\theoremstyle{remark}

\newcommand{\varSpace}{\mathcal{E}}

\usepackage{mathtools}

\usepackage[algo2e, boxruled]{algorithm2e}

\usepackage[T1]{fontenc}
\usepackage{lmodern}
\definecolor{blue}{gray}{0.0}

\begin{document}

    \title{Some Primal-Dual Theory for Subgradient\\ Methods for Strongly Convex Optimization}

    \author{Benjamin Grimmer\footnote{Johns Hopkins University, Department of Applied Mathematics and Statistics, \url{grimmer@jhu.edu}} \qquad Danlin Li\footnote{Johns Hopkins University, Department of Applied Mathematics and Statistics, \url{dli91@alumni.jh.edu}}}

	\date{}
	\maketitle

	\begin{abstract}
        We consider (stochastic) subgradient methods for strongly convex but potentially nonsmooth non-Lipschitz optimization. We provide new equivalent dual descriptions (in the style of dual averaging) for the classic subgradient method, the proximal subgradient method, and the switching subgradient method. These equivalences enable $O(1/T)$ convergence guarantees in terms of both their classic primal gap and a not previously analyzed dual gap for strongly convex optimization. Consequently, our theory provides these classic methods with simple, optimal stopping criteria and optimality certificates at no added computational cost. Our results apply to a wide range of stepsize selections and to non-Lipschitz ill-conditioned problems where the early iterations of the subgradient method may diverge exponentially quickly (a phenomenon which, to the best of our knowledge, no prior works address). Even in the presence of such undesirable behaviors, our theory still ensures and bounds eventual convergence.
	\end{abstract}

    \section{Introduction}

The study of gradient methods for iteratively solving nonsmooth convex minimization problems dates back to as early as the 60s, see~\cite{Shor:Subgradient}. In recent decades, interest in first-order methods for optimization has resurged in popularity throughout data science and machine learning domains due to their low iteration cost and scalability. This has led to the development of a range of new gradient methods~\cite{Nedic2001,Nesterov2005PrimaldualSM,beck2009fast,Yang2015RSGBS,Grimmer2017RadialSM,Johnstone2020,Lu2017RelativeCF,Renegar2022}. 
Here, we instead focus on improving performance guarantees for classic subgradient methods, the natural extensions of gradient descent to nonsmooth settings.

We consider general convex minimization problems of the following form
\begin{equation}
    p_\star = \begin{cases}
        \min & f_0(x) + r(x)\\
        \mathrm{s.t.} & f_s(x) \leq 0 \qquad \forall s=1\dots m
    \end{cases}\label{eq:super-problem}
\end{equation}
where the functions $f_s\colon \mathcal{E}\rightarrow \mathbb{R}$, $s=0,\dots, m,$ are (strongly) convex but may be nonsmooth and not globally Lipschitz continuous and $r\colon \mathcal{E}\rightarrow \mathbb{R}\cup\{\infty\}$ is convex, closed, and simple, all defined over a finite-dimensional Euclidean space $\varSpace$. We will consider iteratively solving problems of this general form via a stochastic switching proximal subgradient method. This general method corresponds to the classic subgradient method when $r=0$ and $m=0$, the proximal subgradient method when $m=0$, and the switching subgradient method of~\cite{Polyak1967} when $r=0$. Formal assumptions and definitions are deferred to Section~\ref{sec:prelim}. 

This work provides equivalent dual descriptions and new primal-dual convergence rates for all of these classic subgradient methods.
Although our theory will be developed for stochastic, non-Lipschitz problems with $r\neq 0$ and $m\neq 0$, we first briefly present our results without these generalities to showcase and introduce the key ideas.

\noindent {\bf The Classic Setting: Primal-Dual Theory for the Subgradient Method.}\\
Supposing $r=0$ and $m=0$, the problem~\eqref{eq:super-problem} reduces to unconstrained minimization of $f_0\colon \varSpace\rightarrow \mathbb{R}$. We assume access to an oracle capable of producing a subgradient at each iteration
$g_0(x)\in \partial f_0(x) := \{  g \mid f_0(y) \geq f_0(x) + \langle g , y-x \rangle \quad \forall y\in\varSpace\}$. Note when $f_0$ is $\mu$-strongly convex (that is, $f(x)-\frac{\mu}{2}\|x\|_2^2$ is convex), each subgradient $g_0(x)\in \partial f_0(x)$, provides a quadratic lower bound
\begin{equation}
    f_0(y) \geq f_0(x) + \langle g_0(x) , y-x \rangle +\frac{\mu}{2}\|y-x\|_2^2 \quad \forall y\in\varSpace \ .
\label{eq:subgradient-lower-bound}
\end{equation}

To develop a primal-dual understanding, we consider the following reformulation
\begin{align*}
    p_\star = \min_{x\in\varSpace} f_0(x) &= \min_{x\in\varSpace} \max_{y\in\varSpace} f_0(y) + \langle g_0(y), x-y \rangle + \frac{\mu}{2}\|x-y\|_2^2\\
    &=\min_{x\in\varSpace}\max_{\substack{y_0,\dots y_T\in\varSpace \\ \lambda_0,\dots \lambda_T\geq 0 \\ \sum_{k=0}^T \lambda_k >0}} \frac{\sum_{k=0}^T \lambda_k\left(f_0(y_k) + \langle g_0(y_k), x-y_k \rangle + \frac{\mu}{2}\|x-y_k\|_2^2\right)}{\sum_{k=0}^T \lambda_k}
\end{align*}
where the first line replaces $f_0$ by the maximum of its quadratic subgradient lower bounds and the second extends this to the maximum combination of such lower bounds.
Any point $x\in\varSpace$ provides a ``primal solution'' with primal gap $f_0(x) - p_\star$. Any collection of points $y_k$ and weights $\lambda_k$ provide a ``dual solution'', which produces a lower bound of $m^{(T)}(x) = \frac{\sum_{k=0}^T \lambda_k\left(f_0(y_k) + \langle g_0(y_k), x-y_k \rangle + \frac{\mu}{2}\|x-y_k\|_2^2\right)}{\sum_{k=0}^T \lambda_k}$ {\color{blue} on $f_0(x)$} and a dual gap $p_\star - \inf m^{(T)}$.

The subgradient method builds a sequence of primal solutions $\{x_k\}$ by repeatedly moving in negative subgradient directions using a sequence of stepsizes $\alpha_k>0$
\begin{equation}
    x_{k+1} = x_k - \alpha_k g_0(x_k) \ .  \label{eq:basic-subgradient-method}
\end{equation}
The method of dual averaging builds dual solutions yielding models $m^{(k)}$ by repeatedly minimizing this lower bounding model (with an optional $\beta_k$ regularization term) and then incorporating a new subgradient $g_{0}(y_{k+1})$ into the next model from the resulting point $y_{k+1}$ with weight $\lambda_{k+1}>0$
\begin{equation}
    \begin{cases}
        m^{(k)}(y) = \dfrac{\sum_{i=0}^k \lambda_i\left(f_0(y_i) + \langle g_0(y_i), y-y_i \rangle + \frac{\mu}{2}\|y-y_i\|_2^2\right)}{\sum_{i=0}^k \lambda_i}\\
        y_{k+1} = \argmin\left\{ m^{(k)}(y) + \dfrac{\beta_k}{2\sum_{i=0}^k \lambda_i}\|y-y_0\|_2^2 \right\}\ . 
    \end{cases}
\end{equation}
If one sets $\mu=0$, this corresponds to the dual averaging method of Nesterov~\cite{Nesterov2005PrimaldualSM}. Other variations of dual averaging from the literature are discussed in Section~\ref{subsec:related}.

In the not strongly convex setting of $\mu=0$, Nesterov~\cite{Nesterov2005PrimaldualSM} showed when $\beta_k=\bar\beta>0$ is constant, these two methods are equivalent whenever $\alpha_k = \lambda_k/\bar\beta$. That is, they produce the same sequence of iterates $x_k=y_k$. Our Theorem~\ref{thm:super-equivalence} extends this equivalence to potentially strongly convex settings, showing these two methods are equivalent whenever $\alpha_k = \lambda_k/(\mu\sum_{i=0}^k\lambda_i +\bar\beta)$ and $\mu + \bar\beta>0$. Such results allow one to equivalently view the classic subgradient method as either iteratively building a primal solution converging down to optimal or building a lower bound converging up to optimal.

This new dual understanding enables us to develop primal-dual convergence theory for the subgradient method. We {\color{blue} derive convergence rate bounds ensuring} the dual model built by the subgradient method converges up to optimal at {\color{blue} at least} the same rate as the {\color{blue} primal} iterates {\color{blue} are guaranteed to} converge down. For example, if $\bar\beta=0$ and $\|g_0(x_k)\|_2 \leq M$ uniformly, a special case of our Theorem~\ref{thm:super-rate} implies for any selection of $\alpha_k$ and $\lambda_k$ with $\alpha_k = \lambda_k/(\mu\sum_{i=0}^k\lambda_i)$, the subgradient method (or equivalently dual averaging) has primal gap, dual gap, and distance to optimality all converge with
\begin{equation}
    \underbrace{\frac{\sum_{k=0}^T \lambda_k f(x_k)}{\sum_{k=0}^T \lambda_k} - p_\star}_{\substack{\mathtt{Primal\ Gap}}} + \underbrace{p_\star - \inf m^{(T)}}_{\substack{\mathtt{Dual\ Gap}}} + \underbrace{\frac{\mu}{2}\|x_{T+1}-x_\mathtt{OPT}\|_2^2}_{\substack{\mathtt{Distance\ To\ Optimal}}} \leq M^2 \frac{\sum_{k=0}^T \lambda_k\alpha_k}{\sum_{k=0}^T \lambda_k} \ . \label{eq:easy-bound}
\end{equation}
Selecting dual weights $\lambda_k = k+1$, corresponding to primal stepsizes $\alpha_k = 2/\mu(k+2)$ as considered in~\cite[Section 3.2]{Lacoste-Julien2012}, this recovers and extends their optimal primal rate $O(M^2/\mu T)$ to have a matching dual term. This dual theory enables a computable optimal stopping criteria (assuming $\mu$ is known) as a primal-dual gap $\frac{\sum_{k=0}^T \lambda_k f_0(x_k)}{\sum_{k=0}^T \lambda_k} - \inf m^{(T)}$ less than $\epsilon$ occurs within $O(1/\epsilon)$ steps ensuring both primal and dual accuracies of at least $\epsilon$. To the best of our knowledge, no such criterion has been known, even for the classic subgradient method.

Given our primal-dual equivalence, our results can be equally seen as providing primal-dual convergence guarantees for dual averaging. In this sense, we improve the prior best dual averaging theory due to Deng et al.~\cite[Corollary 8]{Deng2015RandomizedBS} who showed a primal rate of $O\left(M^2/\mu T\right)$
when $\lambda_k = k+1$ and $\beta_k =\mu(T+2)$ is constant. Note such rates are optimal (by the example presented in~\cite{Nesterov-Introductory-Lectures}), meaning no faster objective gap convergence rate in terms of any of $M, \mu, T$ can be achieved.

\subsection{Our Contributions}
This work develops primal-dual equivalences and convergence theory beyond the above classic subgradient method setting. We consider the general problem~\eqref{eq:super-problem} including an additive composite objective, functional constraints, and stochastic subgradients.
\begin{itemize}
    \item {\bf Dual Equivalences and Primal-Dual Convergence Theory.} Our theory considers a {\it Stochastic Switching Proximal Subgradient Method} for the general problem class~\eqref{eq:super-problem}, which, as special cases, captures projected, proximal, and switching subgradient methods as well as gradient descent. We introduce a new dual averaging method for this general problem class, {\it Stochastic Lagrangian Proximal Dual Averaging}, which our Theorem~\ref{thm:super-equivalence} shows is equivalent to the stochastic switching proximal subgradient method under proper selection of primal stepsizes $\alpha_k$ and dual weights $\lambda_k$ and $\beta_k$. From this equivalence, our Theorem~\ref{thm:super-rate} presents new primal-dual convergence rate guarantees for these general methods.

    \item {\bf Computable Stopping Criteria.} Our theory identifies dual certificates implicitly built by the range of considered subgradient methods. These certificates enable new computable stopping criteria (assuming $\mu$ is known), halting once the primal-dual gap is at most a target accuracy $\epsilon$. The associated Lagrange multipliers may further be valuable when the subgradient method is used as a subroutine of a larger computation. 
    
    \item {\bf New Non-Lipschitz Analysis Bounds for Early Divergence Phenomena.}
    Often, nonsmooth optimization analysis focuses on Lipschitz continuous functions. Such theory is limited to functions that asymptotically grow at most linearly. Our analysis uses a non-Lipschitz model, allowing up to quadratic growth. By doing so, our theory provides new linear primal-dual convergence guarantees for gradient descent in smooth optimization and new guarantees for minimizing a sum $f_0=h_1+h_2$ with smooth $h_1$ and nonsmooth but Lipschitz $h_2$, which is overall neither Lipschitz nor smooth. Numerics showcasing highly non-monotone behaviors of subgradient methods on such problems are shown in Section~\ref{sec:numerics}, where our theory still provides reasonably accurate predictions.
\end{itemize}

\noindent {\bf Outline.} Section~\ref{sec:prelim} introduces the considered primal and dual subgradient methods, related literature, and the assumptions needed for our theory. Section~\ref{sec:results} states and proves our equivalence between these primal and dual perspectives and our improved primal-dual guarantees. Finally, Section~\ref{sec:numerics} concludes with some numerical validation.
    \section{Preliminaries and Algorithm Definitions}
\label{sec:prelim}
Recall we are interested in the family of problems
\begin{equation*}
    \begin{cases}
        \min & f_0(x) + r(x)\\
        \mathrm{s.t.} & f_s(x) \leq 0 \qquad \forall s=1\dots m \ . 
    \end{cases}
\end{equation*}
This extends the previously discussed classic subgradient method setting in three ways.

First, we allow for additive composite objectives $f_0+r$ for any {\color{blue} proper} closed convex $r\colon \varSpace\rightarrow \mathbb{R}\cup\{\infty\}$. We address this added term by assuming $r$ is sufficiently simple that its proximal operator
$ \mathrm{prox}_{\alpha_k, r}(z) := \argmin_x\left\{r(x) + \frac{1}{2\alpha_k}\|x-z\|_2^2\right\} $
can be evaluated at each iteration. For example, if $r$ is an indicator function for a closed convex constraint set, its proximal operator corresponds to projection onto that set. Since $r(x) + \frac{1}{2\alpha_k}\|x-z\|_2^2$ is strongly convex, it has a unique minimizer described by the following equivalence
\begin{equation}
    z_+ = \mathrm{prox}_{\alpha_k, r}(z) \quad \iff \quad \frac{1}{\alpha_k}(z-z_+)\in \partial r(z_+) \ . \label{eq:proximal-equivalence}
\end{equation}
Second, we allow for $m$ strongly convex functional constraints $f_s(x)\leq 0$ for $s=1,\dots, m$. We address these added terms by ``switching'': Given a current iterate $x$, only one function $f_{s(x)}$ will be considered in that iteration. This function will be chosen as $s(x)=0$ if $x$ is feasible (that is, $f_s(x)\leq 0$ for all $s=1\dots m$), otherwise $s(x)$ can be chosen generically as any violated constraint $f_{s(x)}(x)>0$.
Third, we allow for stochasticity in our subgradient oracles for each function, denoted by $g_s(x;\xi)$ such that $\mathbb{E}_\xi g_{s}(x;\xi) \in \partial f_{s}(x)$. Note this trivially captures deterministic methods by selecting $g_s(x;\xi)$ constant with respect to $\xi$.

Below, we introduce our considered primal and dual subgradient methods for solving such problems. We introduce these using disjoint notations but will in Theorem~\ref{thm:super-equivalence} show these are, in fact, the same algorithm in that their iterate sequences are identical.

\noindent {\bf A General Primal Subgradient Method.}
As a primal algorithm, consider the {\it Stochastic Switching Proximal Subgradient Method} with stepsizes $\alpha_k>0$ and sequence of iterates $\{x_k\}$ defined by
\begin{equation}\label{eq:super-primal-method}
    x_{k+1} = \begin{cases}
        \mathrm{prox}_{\alpha_k, r} (x_k-\alpha_k g_{0}(x_k;\xi_k)) & \text{if $x_k$ is feasible}\\
        x_k-\alpha_k g_{s(x_k)}(x_k;\xi_k) & \text{otherwise.}        
    \end{cases}
\end{equation}
for i.i.d.~sampled $\xi_k$. Throughout, we always assume $\alpha_k>0$, strictly.
When $m=0$, this is the standard (stochastic) proximal subgradient method; when $r=0$, this is the (stochastic) switching subgradient method. Note when $m>0$, only limited stochasticity can be allowed since an exact determination of feasibility is required to decide the switching variable $s(x_k)$.

%For ease of notation, we let $I(s,k)=\{ i \mid s(y_i) = s, i\leq k\}$ denote the set of all iterations where the switching selects $f_s$ up to iteration $k$. Further, when $s(y_k)=0$ (i.e., $k\in I(0,k)$), we let $n_{k+1}:=(x_k-\alpha_k g_k - x_{k+1})/\alpha_k \in \partial r(x_{k+1})$ denote the subgradient of $r$ determined by~\eqref{eq:proximal-subgradient}.

\noindent {\bf A General Dual Subgradient Method.} To give a dual approach to solving~\eqref{eq:super-problem}, consider the following equivalent minimax formulation
\begin{align*}
    p_\star &= \min_{x\in\varSpace}\max_{u_1,\dots u_s\geq 0} f_0(x) + \sum_{s=1}^m u_sf_s(x) + r(x) \\
    &=\min_{x\in\varSpace}\max_{\substack{y_0,\dots, y_{T}\in\varSpace\\ \lambda_0,\dots, \lambda_{T-1}\geq 0 \\ \sum_{k<T\colon s(y_k)=0} \lambda_k >0\\ n_{k+1}\in\partial r(y_{k+1})}} \mathbb{E}_{\xi}\Bigg(\sum_{s=0}^m\sum_{\substack{k<T \\ s(y_k)=s}} \lambda_k\left(f_{s(y_k)}(y_k) + \langle g_{s(y_k)}(y_k; \xi_k), x-y_k \rangle + \frac{\mu}{2}\|x-y_k\|_2^2\right)\\
    &\qquad\qquad\qquad\qquad\qquad\qquad +\sum_{\substack{k<T\\ s(y_k)=0}} \lambda_k\left(r(y_{k+1}) + \langle n_{k+1}, x - y_{k+1}\rangle\right)\Bigg)\frac{1}{\sum_{k<T\colon s(y_k)=0} \lambda_k}
\end{align*}
where the first equality is the standard primal Lagrangian formulation and the second equality replaces each function by a combination of its lower bounds. Hence, any selection of values for $y_k,\lambda_k, n_k$ gives a dual solution and a lower bound on $p_\star$.

As a dual algorithm, consider the following {\it Stochastic Lagrangian Proximal Dual Averaging} with dual weights $\lambda_k>0$ and regularization parameters $\beta_k\geq 0$. Throughout, we always assume $\lambda_k>0$, strictly. Its sequence of iterates $\{y_k\}$ based on i.i.d.~sampled $\xi_k$ is defined as follows: At iteration $k$, construct the following (unnormalized) aggregations for each function based on the previous $k-1$ iterations as
\begin{align*}
    F^{(k-1)}_s(y) &:= \sum_{i<k \colon  s(y_i)=s} \lambda_i \left(f_{s(y_i)}(y_i) + \langle g_{s(y_i)}(y_i;\xi_i), y-y_i\rangle + \frac{\mu}{2}\|y-y_i\|_2^2\right)\\
    R^{(k-1)}(y) &:= \sum_{i<k \colon  s(y_i)=0} \lambda_i \left(r(y_{i+1}) + \langle n_{i+1}, y-y_{i+1}\rangle\right)\\
    M^{(k-1)}(y) &:= \sum_{s=0}^m F^{(k-1)}_s(y) + R^{(k-1)}(y)
\end{align*}
where $n_{i+1} = \frac{-1}{\lambda_i}(\nabla M^{(i-1)}(y_{i+1}) + \lambda_i(g_0(y_i;\xi_i) +\mu(y_{i+1}-y_i)) + \beta_i(y_{i+1} - y_0)) \in \partial r(y_{i+1})$ (see Lemma~\ref{lem:Mk-minimization} for verification of this subdifferential containment).
At iteration $k=0$, these empty summations are understood to take value zero.
Then, based on the switching selection $s(y_k)$, a new weighted model is constructed as
$$ U^{(k)}(y) := \begin{cases} \lambda_k\left(f_{0}(y_k) + \langle g_0(y_k;\xi_k), y-y_k\rangle + \frac{\mu}{2}\|y-y_k\|_2^2 + r(y)\right) & \text{if $y_k$ is feasible}\\
\lambda_k\left(f_{s(y_k)}(y_k) + \langle g_{s(y_k)}(y_k;\xi_k), y-y_k\rangle + \frac{\mu}{2}\|y-y_k\|_2^2\right) & \text{otherwise.}
\end{cases} $$
The Lagrangian proximal dual averaging method then iterates by minimizing the aggregation of past models $M^{(k-1)}$ plus the new model $U^{(k)}$ (and an optional extra regularization term)
\begin{equation}\label{eq:super-dual-method}
    y_{k+1} = \argmin\left\{M^{(k-1)}(y) + U^{(k)}(y) +\frac{\beta_k}{2}\|y-y_0\|_2^2 \right\} \ .
\end{equation}

Note our definitions for the updated aggregate model $M^{(k)}$ are chosen such that $y_{k+1}$ will also be the unique minimizer of $M^{(k)}(y)+\frac{\beta_k}{2}\|y-y_0\|_2^2$. 
\begin{lemma}\label{lem:Mk-minimization}
    $y_{k+1}$ is the unique minimizer of $M^{(k)}(y)+\frac{\beta_k}{2}\|y-y_0\|_2^2$.
\end{lemma}
\begin{proof}
% {\color{blue}
First note that when $y_k$ is infeasible (and so $s(y_k)\neq 0$), $M^{(k)} = M^{(k-1)} + U^{(k)}$. From this, the result is immediate.
Now consider when $y_k$ is feasible (and so $s(y_k)=0$). Since $y_{k+1}$ is the unique minimizer of $M^{(k-1)}(y)+U^{(k)}(y)+\frac{\beta_k}{2}\|y-y_0\|_2^2$, the necessary and sufficient optimality condition ensures $y_{k+1}$ is the unique solution to
$$ 0 \in \nabla M^{(k-1)}(y_{k+1}) + \lambda_k(g_0(y_k; \xi_k) + \mu(y_{k+1}-y_k) + \partial r(y_{k+1})) + \beta_k(y_{k+1}-y_0) \ . $$
Rewriting this as $ \nabla M^{(k-1)}(y_{k+1}) + \lambda_k (g_0(y_k; \xi_k)+\mu(y_{k+1}-y_k)) + \beta_k(y_{k+1}-y_0) \in -\lambda_k \partial r(y_{k+1})$, we see that $n_{k+1}$ is the element of $\partial r(y_{k+1})$ certifying the minimization's optimality. Therefore
\begin{align*}
    0 &= \nabla M^{(k-1)}(y_{k+1}) + \lambda_k(g_0(y_k; \xi_k) + \mu(y_{k+1}-y_k)+n_{k+1}) + \beta_k(y_{k+1}-y_0)\\
    &= \nabla M^{(k)}(y_{k+1}) + \beta_k(y_{k+1}-y_0)
\end{align*}
and so $y_{k+1}$ is a the unique minimizer of $M^{(k)}(y)+\frac{\beta_k}{2}\|y-y_0\|_2^2$. 

\end{proof}

Note whenever $s(y_k)=0$, the step~\eqref{eq:super-dual-method} corresponds to minimizing $r$ plus a simple quadratic. This amounts to computing a proximal operator for $r$ and so is within the assumed computational oracle model. Whenever $s(y_k)\neq 0$, the step~\eqref{eq:super-dual-method} minimizes a simple quadratic and can be done in closed form. The following easily verifiable lemmas provide a simple way to maintain the minimum of the aggregate quadratic model.

\begin{lemma}\label{lem:standard-quadratic}
    Any quadratic function of the form $Q(z) = c + \langle d,z\rangle + \frac{b}{2}\|z-\hat z\|^2_2$ with $b>0$ is equal to $Q(z) = \min Q + \frac{b}{2}\|z - \argmin Q\|^2_2$.
\end{lemma}

\begin{lemma}\label{lem:quad-calc}
    The sum of quadratics $Q_i(z) = a_i + \frac{b_i}{2}\|z-z_i\|_2^2$ for $i=1,2$ with $b_i > 0$ equals
    $$ (Q_1+Q_2)(z) = a_{1+2} + \frac{b_{1+2}}{2}\|z-z_{1+2}\|_2^2$$
    where $a_{1+2} = a_1 + a_2 + \frac{b_1b_2}{2(b_1+b_2)}\|z_1-z_2\|_2^2$, $b_{1+2} = b_1+b_2$, and $z_{1+2} = \frac{b_1}{b_1+b_2}z_1 + \frac{b_2}{b_1+b_2}z_2$.
\end{lemma}

\subsection{Related Work}\label{subsec:related}
\noindent {\bf More General Distance Terms in Dual Minimization.} Nesterov's original development~\cite{Nesterov2005PrimaldualSM} and most of the subsequent literature have considered a slightly different model for dual averaging than discussed here. To the best of our knowledge, no previous dual averaging methods handled functional constraints. Instead, they fix $m=0$. Prior works have primarily fixed $\mu=0$ (not utilizing the quadratic improvement in lower bound quality from strong convexity) but allowed a more generic distance function in the second term of dual averaging's objective to be minimized at each step. The ``standard'' dual averaging iteration for unconstrained minimization of $f_0$ is then
\begin{equation}\label{eq:distance-dual-method}
    \begin{cases}
    m^{(k)}(x) = \frac{\sum_{i=0}^k \lambda_i \left(f(x_i) + \langle g_0(x_i;\xi_i), x-x_i \rangle\right)}{\sum_{i=0}^k \lambda_i}\\
    x_{k+1} = \argmin\left\{m^{(k)}(x) + \frac{\beta_{k}}{2\sum_{i=0}^k \lambda_i}d(x)\right\}
    \end{cases}
\end{equation}
for any $\rho$-strongly convex $d(x)$. Our equivalent dual perspective fundamentally relies on these improvements in the subgradient lower bounds and the distance function both being quadratics $\|x-x_i\|_2^2$ and, as a result, are directly relatable. We do not expect our theory to generalize easily for more generic distance functions.

\noindent {\bf Regularized Dual Averaging.} A closely related method to the proximal subgradient method is {\it Regularized Dual Averaging} proposed by Xiao~\cite{Xiao2009} and further extended by~\cite{Chen2012,Jumutc2014ReweightedL2,Deng2015RandomizedBS,Siegel2019ExtendedRD,Tan2020}. This method applies to unconstrained additive composite problems minimizing $f_0+r$ by iterating
\begin{equation}
    \begin{cases}
        m^{(k)}(x) = \frac{\sum_{i=0}^k \lambda_i \left(f_0(x_i) + \langle  g_0(x_i;\xi_i), x-x_i\rangle  + \frac{\mu}{2}\|x-x_i\|_2^2\right)}{\sum_{i=0}^k \lambda_i} + r(x) \\
        x_{k+1} \in \argmin\left\{ m^{(k)}(x) + \frac{\beta_{k}}{2 \sum_{i=0}^k \lambda_i}\|x-x_0\|_2^2\right\} \ .
    \end{cases} \label{eq:regularized-dual-avg}
\end{equation}
This method's original development in~\cite{Xiao2009} fixed $\mu=0$ but allowed for a more general distance function, as discussed above. Based on our Theorem~\ref{thm:super-equivalence}, regularized dual averaging can be seen as a natural improvement on the proximal subgradient method. Regularized Dual Averaging utilizes $r$ entirely in its model function, whereas our equivalent dual description of the proximal subgradient method (\eqref{eq:super-dual-method} specialized to this case, $m=0$) uses the mixture of subgradient lower bounds and $r$, iterating
\begin{equation}
    \begin{cases}
        m^{(k)}(x) = \frac{\sum_{i=0}^k \lambda_i \left(f_0(x_i) + \langle  g_0(x_i;\xi_i), x-x_i\rangle  + \frac{\mu}{2}\|x-x_i\|_2^2\right)}{\sum_{i=0}^k \lambda_i} + \frac{\sum_{i=0}^{k-1} \lambda_i \left(r(x_{i+1}) + \langle  n_{i+1}, x-x_{i+1}\rangle\right) + \lambda_k r(x)}{\sum_{i=0}^k \lambda_i} \\
        x_{k+1} \in \argmin\left\{ m^{(k)}(x) + \frac{\beta_{k}}{2 \sum_{i=0}^k \lambda_i}\|x-x_0\|_2^2\right\} \ .
    \end{cases} \label{eq:proximal-dual-avg}
\end{equation}

\noindent {\bf Switching Subgradient Method Guarantees.} Convergence rate guarantees for the switching subgradient method of~\cite{Polyak1967} have been extensively studied for convex Lipschitz minimization~\cite{Lan2020Switching,Bayandina2018,Alkousa2020,Titov2020} and more recently for non-Lipschitz settings~\cite{jia2023firstorder}. Our theory extends this prior theory to give matching dual bounds and identifies Lagrange multipliers $u_s = \frac{\sum_{k: s(y_k) = s} \lambda_k}{\sum_{k: s(y_k) = 0} \lambda_k}$ implicitly built by this classic method (at no added cost). Recently, nonconvex guarantees were developed by~\cite{huang2023singleloop} but are beyond our scope.
% When $\mu>0$, the dual averaging method~\eqref{eq:dual-method} is closely related to the regularized dual averaging methods considered by~\cite{}. To see this relationship, one can rewrite $f$ as a convex term $\bar f = f - \frac{\mu}{2}\|\cdot - x_0\|_2^2$ plus regularization $\frac{\mu}{2}\|\cdot - x_0\|_2^2$. Then subgradients of $\bar f$ are given by $\bar g_k = g_k - \mu (x_k-x_0)$ and simple manipulation of~\eqref{eq:dual-method} shows it equivalent to
% \begin{equation}\label{eq:reg-dual-method}
%     x_{k+1} = \argmin_{x\in \varSpace}\left\{\sum_{i=0}^k \lambda_i\left(\langle \bar g_i, x\rangle  + \frac{\mu}{2}\|x-x_0\|^2_2)\right) + \frac{\beta_{k+1}}{2}\|x-x_0\|^2_2\right\} \ .
% \end{equation}

\noindent {\bf Convex Conjugate-type Convergence Analysis.} The recent series of works of Pe\~na and Gutman~\cite{Pena2017,Gutman2018,Gutman2022} developed unified convergence guarantees for convex optimization ($\mu=0$) for a range of first-order methods from accelerated smooth methods to nonsmooth Bregman and conditional subgradient methods. Beyond just showing convergence of the objective gap, these works showed convergence of perturbed primal-dual quantities based on aggregating (sub)gradient information. This work shares a similar spirit but addresses the setting of $\mu>0$. Strong convexity ensures our non-perturbed dual gaps are finite, a necessity for our theory.

The recent work of Diakonikolas and Orecchia~\cite{Diakonikolas2019ContinuousDualGap} developed first-order methods by discretizing continuous-time dynamical systems with decreasing gaps between aggregated upper and lower bounds on optimality. This technique is able to recover dual averaging among many other standard first-order methods. Although they only provide primal guarantees, their approach may be extendable to bound dual gaps.

\noindent {\bf Prior Primal Weighted Averaging Analysis.} The value and importance of returning a carefully chosen weighted combination $\sum_{k=0}^{T-1}\sigma_k x_k$ of subgradient method iterates has been studied by several prior works. Rakhlin et al.~\cite[Theorem 5]{Rakhlin2012} showed uniformly averaging the last $q\in(0,1)$ fraction of iterates (called $q$-suffix averaging) can lead to an optimal $O(1/T)$ primal convergence rate for strongly convex minimization. Shamir and Zhang~\cite[Theorem 3 and 4]{Shamir2013} improved this theory and additionally showed polynomial weightings yield the same optimal rate with less computational overhead. The $s^k$-stepsize rule developed by Gustavsson et al.~\cite[Section 3.3]{Gustavsson2015} builds substantial theory for such polynomial stepsizes and weightings $\sigma_k=(k+1)^p$. The $\sigma_k = k+1$ choice of~\cite{Lacoste-Julien2012} amounts to the simplest polynomial weighting choice. Our theory provides a novel insight into the source of these iterate aggregation weights: primal-dual guarantees hold for any stepsizes $\alpha_k$ if the averaging used is proportional to the stepsize's corresponding dual weights $\sigma_k\propto \lambda_k$.

\noindent {\bf Alternative Lagrangian Dual Averaging Approaches.} Our proposed Lagrangian Proximal Dual Averaging Method~\eqref{eq:super-dual-method} implicitly sets the dual multipliers based on the frequency/total weight of steps taken on each constraint function, after which the iteration amounts to repeated model minimization. Alternatively, one could apply dual averaging to the Lagrangian minimax problem directly. Such an approach is proposed and analyzed by Metel and Takeda~\cite{Metel2021}.

\subsection{Assumptions for our Convergence Theory}
Our primal-dual convergence rate theory relies on three assumptions. Our first two assumptions are standard, strong convexity and the existence of a Slater point.

\noindent {\bf Assumption A.} The functions $f_s$ for $s=0\dots m$ are each $\mu>0$-strongly convex.\\
\noindent {\bf Assumption B.} There exists some $x_\mathtt{SL} \in \mathrm{dom\ }\partial r$ with $f_s(x_\mathtt{SL}) <0$ for all 
 $s = 1\dots m$.

Note strong convexity ensures there exists a unique minimizer $x_\mathtt{OPT} \in \mathrm{dom\ }\partial r$ of~\eqref{eq:super-problem}. These two points $x_\mathtt{OPT},x_\mathtt{SL} \in \mathrm{dom\ }\partial r$ serve as important references that our analysis is done with respect to. We fix two subgradients of $r$ at these points: The subgradient $n_{x_\mathtt{OPT}}\in \partial r(x_\mathtt{OPT})$ is chosen such that $f_0(x) + r(x_\mathtt{OPT}) + \langle n_{x_\mathtt{OPT}}, x-x_\mathtt{OPT}\rangle$ is minimized over $f_s(x)\leq 0$ at $x_\mathtt{OPT}$ for all $ s=1\ldots m$. The subgradient $n_{x_\mathtt{SL}} \in \partial r(x_\mathtt{SL})$ can be chosen freely. %When $r$ is an indicator function for a closed convex set, $n_{x_\mathtt{SL}}=0$ is a nice simplifying choice.

These two reference subgradients of $r$ facilitate considering two lower bounds of the objective $f_0+r$ for either $y\in\{x_{\mathtt{OPT}}, x_{\mathtt{SL}}\}$, denoted by
$$ h_{y}(x) := f_0(x) + r(y) + \langle n_{y}, x-y\rangle \ . $$
%If $r$ is an indicator function, selecting $n_{x_\mathtt{SL}}=0$ simplifies $h_{x_\mathtt{SL}} = f_0$.
At each iteration $k$, we denote the relative difference between $x_k$ and $y\in\{x_{\mathtt{OPT}}, x_{\mathtt{SL}}\}$ in (relaxed) objective value or feasibility on the selected constraint function $f_{s(x_k)}$ by
\begin{align}
    \delta_k(y) &:=\begin{cases}
        h_y(x_k) - h_y(y) & \mathrm{if\ } x_k \mathrm{\ is\ feasible}\\
        f_{s(x_k)}(x_k) - f_{s(x_k)}(y) & \mathrm{otherwise.}
    \end{cases}\label{eq:delta_k-definition}
\end{align}
Note $\delta_k(y)$ is always finite since the real-valued objective function lower bound $h_y$ is used instead of $f_0+r$ which takes value in the extended reals. Indeed, consider $r$ as an indicator function for a simple constraint set. This set is projected onto at each iteration where $x_k$ is feasible for all of the functional constraints, ensuring $x_{k+1}$ is feasible for the simple constraint. We cannot, however, guarantee that $x_k$ satisfies the simple constraint, and so $r(x_k)$ may be infinite.

The sign of $\delta_k(y)$ may vary. If $y = x_{\mathtt{OPT}}$, then $\delta_k(x_\mathtt{OPT})$ is nonnegative, being lower bounded by the level of suboptimality or current infeasibility
$$ \delta_k(x_\mathtt{OPT}) \geq \begin{cases}
        h_{x_\mathtt{OPT}}(x_k) - p_\star & \mathrm{if\ } x_k \mathrm{\ is\ feasible}\\
        f_{s(x_k)}(x_k) & \mathrm{otherwise}
    \end{cases} \ \geq 0 \ . $$
When $y = x_{\mathtt{SL}}$ and $x_k$ is feasible, $\delta_k(y)$ may be negative but is bounded below by
$$ \delta_k(x_\mathtt{SL}) \geq \inf h_{x_{\mathtt{SL}}} -h_{x_{\mathtt{SL}}}(x_{\mathtt{SL}}) > -\infty \ . $$ When $y = x_{\mathtt{SL}}$ and $x_k$ is infeasible, $\delta_k(y)$ is strictly positive, being bounded below by $$\delta_k(x_\mathtt{SL}) \geq 0-\max_{s=1\dots m} f_s(x_\mathtt{SL}) >0 \ . $$
A common third assumption used in the analysis of subgradient methods is the uniform boundedness of subgradients. However, if this holds everywhere, the objective must be uniformly Lipschitz continuous, implying it asymptotically grows at most linearly. Contradicting this, strong convexity implies it grows at least quadratically. To avoid such incongruences and to include combinations of smooth and nonsmooth optimization, we consider a more general model than Lipschitz continuity similar to that previously considered in~\cite[Section 1.2]{Grimmer2019}, allowing for quadratic growth.

\noindent {\bf Assumption C.} For both $y\in\{x_{\mathtt{OPT}}, x_{\mathtt{SL}}\}$, there exist constants $L_0$, $L_1{\color{blue}\geq 0}$ such that every iterate $x_k$ has
\begin{equation}\label{eq:NonLipschitz-condition}
    \begin{cases}
        x_k \mathrm{\ is\ feasible} & \implies \mathbb{E}_{\xi_k}\|g_0(x_k;\xi_k) + n_{y}\|_2^2 \leq L_0^2 + L_1\delta_k(y)\\
        x_k \mathrm{\ is\ not\ feasible} & \implies \mathbb{E}_{\xi_k}\|g_{s(x_k)}(x_k;\xi_k)\|_2^2 \leq L_0^2 + L_1\delta_k(y) \ . 
    \end{cases}
\end{equation}
This assumption captures several common settings. The following lemma shows this condition holds for any additive combination of nonsmooth Lipschitz and smooth settings with bounded variance in the stochastic subgradient oracles. % \begin{lemma}
%     If each $f_s$ for $s=0,\dots, m$ is uniformly $M$-Lipschitz continuous and $g_{s(x_k)}(x_k;\xi_k)$ has variance uniformly bounded by $\sigma^2$, then Assumption C holds with $L_0^2 =2M^2 + 2\max\{\|n_{x_\mathtt{OPT}}\|_2^2, \|n_{x_\mathtt{SL}}\|\|_2^2\} + \sigma^2$ and $L_1=0$.
% \end{lemma}
% \begin{proof}
%     Supposing $x_k$ is feasible and letting $g =  \mathbb{E}_{\xi_k} g_0(x_k;\xi_k)$, one has
%     \begin{align*}
%         \mathbb{E}_{\xi_k}\|g_0(x_k;\xi_k) + n_{y}\|_2^2 &= \| g + n_{y}\|_2^2 + \mathbb{E}_{\xi_k}\|g_0(x_k;\xi_k) -g\|_2^2 \\ 
%         &\leq 2M^2 + 2\|n_y\|_2^2 + \sigma^2 \ . 
%     \end{align*}
%     Similarly, if $x$ is infeasible, $\mathbb{E}_{\xi_k}\|g_{s(x_k)}(x_k;\xi_k)\|_2^2 \leq M^2 +\sigma^2$.
% \end{proof}
\begin{lemma}\label{lem:AssumptionC-verification}
    If there exists functions $f^{(1)}_s, f^{(2)}_s$ such that each $f_s = f_s^{(1)} + f^{(2)}_s$ for $s=0,\dots, m$ where $f^{(1)}_s$ is uniformly $M$-Lipschitz and $f^{(2)}_s$ has uniformly $L$-Lipschitz gradient and $g_{s(x_k)}(x_k;\xi_k)$ has variance uniformly bounded by $\sigma^2$, then Assumption C holds.
\end{lemma}
The proof of this lemma is deferred to the appendix where the explicit constants $L_0$ and $L_1$ can be found. {\color{blue} This established that in the standard nonsmooth optimization setting where each $f_s$ is uniformly Lipschitz, Assumption C holds with $L_1=0$. Further, Assumption C also holds in the standard smooth optimization setting where each $f_s$ has uniformly $L$-Lipschitz gradient.} 
Section~\ref{sec:numerics} gives an illustrative numerical example of this form where Assumption C holds despite the objective being neither Lipschitz nor smooth.

Note allowing non-Lipschitz objectives allows a range of undesirable ``bad'' behaviors to occur. It allows the early iterations of the subgradient method to diverge exponentially. For example, consider minimizing $f(u,v) = 50u^2 + 0.5v^2$, which has $\mu=1,L_0=0,L_1=200$, with the subgradient method~\eqref{eq:basic-subgradient-method} initialized with $x_0=(1,0)$ and $\alpha_k=2/\mu(k+2)$ (corresponding to $\lambda_k=k+1,\beta_k=0$). For the first one hundred iterations, the size of $x_k$ grows exponentially, peaking with $\|x_{100}\|_2$ just over $10^{56}$, after which it converges monotonically to $f$'s minimizer. Despite such instances existing, our theory shows that even if $x_k$ diverges in its early iterations, it will always subsequently converge at least at rate $O(1/T)$. To the best of our knowledge, no existing analysis of subgradient methods or dual averaging addresses this phenomenon.

To understand and bound such behaviors, we introduce the following two constants, dependent on the choice of stepsizes $\alpha_k$ and associated dual weights $\lambda_k$,
\begin{align}
    T_0 &:= \sup\left\{ k\in\{0,1,2\dots\} \mid L_1\alpha_k > 1\right\} \ ,\label{eq:T0} \\
    C_0 &:= \sum_{k=0}^{T_0}\lambda_k \max\left\{L_1 \alpha_k - 1, 0 \right\}\max\{\mathbb{E}_\xi\delta_k(x_\mathtt{OPT}), \mathbb{E}_\xi\delta_k(x_\mathtt{SL})\} \ . \label{eq:C0}
\end{align}
Observe that $T_0$, and hence $C_0$, is bounded if $\alpha_k$ is eventually always less {\color{blue} or equal to} $1/L_1$, capturing all stepsize policies with $\alpha_k\rightarrow 0$. Despite being bounded, the constant $C_0$ can be exponentially large in $T_0$. The toy example considered above has $T_0=397$ and $C_0 > 10^{112}$. Such potentially exponential-sized constants can be avoided through careful stepsize selection as if one selects $\alpha_k \leq 1/L_1$ for all $k$ as then $T_0=-\infty$ and $C_0 = 0$. In particular, under the classic assumption that subgradients seen are uniformly bounded, Assumption C holds with $L_1=0$ and so $C_0=0$. Regardless, our convergence guarantees apply whenever $C_0$ is finite. As a natural consequence of our main convergence analysis, we find the rate $x_k$ can diverge, and hence the constant $C_0$, are at most exponential in $T_0$ (see Proposition~\ref{prop:exp-bound}).

\section{Primal-Dual Equivalence and Convergence Analysis} \label{sec:results}
In this section, we show the considered primal switching proximal method~\eqref{eq:super-primal-method} and dual Lagrangian proximal method~\eqref{eq:super-dual-method} are equivalent and subsequently state and prove new primal-dual convergence guarantees for these methods.

\begin{theorem} \label{thm:super-equivalence}
    Let $\{x_k\}$ be the sequence of iterates of the primal method~\eqref{eq:super-primal-method} with stepsizes $\alpha_k>0$ and $\{y_k\}$ be the sequence of iterates of the dual method~\eqref{eq:super-dual-method} for some $\lambda_k,\beta_k,\mu\geq 0$. If $x_0=y_0$, $\beta_k=\bar\beta\geq 0$ is constant, and
    \begin{equation}
        \alpha_k = \frac{\lambda_k}{\mu\sum_{i=0}^k\lambda_i + \bar\beta}\ , \label{eq:alpha-lambda}
    \end{equation} 
    then these methods are equivalent, that is $x_k=y_k$.
\end{theorem}
\begin{proof}
        We prove this by inductively showing that $x_k=y_k$ and for any iteration with $s(x_k)=0$ that $n_{k+1} = \frac{1}{\alpha_k}(x_k-\alpha_k g_{0}(x_k;\xi_k) -x_{k+1})$ (i.e., the subgradients of $r$  produced by the dual method are exactly those produced by the primal method via~\eqref{eq:proximal-equivalence}).
    By assumption, $x_0=y_0$. Suppose for induction that $x_i=y_i$ for all $i=0,\dots,k$ and for all $i=0,\dots,k-1$ with $s(x_i)=0$ that $n_{i+1} = \frac{1}{\alpha_i}(x_i-\alpha_i g_{0}(x_i;\xi_i) -x_{i+1})$. Let $\bar g_i = g_0(y_i;\xi_i) + n_{i+1}$ if $y_i$ is feasible and $g_{s(y_i)}(y_i;\xi_i)$ otherwise. For each $i<k$, we claim $x_{i+1} =  x_i - \alpha_i \bar g_i$. If $x_i$ is infeasible, this is immediate. If $x_i$ is feasible, {\color{blue} rearrangement of our inductive hypothesis of on $n_{i+1}$ gives the claim}. Inductively, we conclude for $i<k$ that $ x_{i+1} - x_0 = - \sum_{t=0}^i \alpha_t \bar g_t. $
    
    By Lemma~\ref{lem:Mk-minimization}, $y_{k+1}$ is the unique solution to
    $\sum_{i=0}^k \lambda_i\left(\bar g_i + \mu(y_{k+1}-y_i) \right) + \bar\beta(y_{k+1}-y_0) = 0. $
    Rearranging and simplifying this, it follows that
    \begin{align*}
    y_{k+1} &= y_0 + \frac{\sum_{i=0}^k\lambda_i\mu(y_i-y_0)-\sum_{i=0}^k\lambda_i\bar g_i}{\sum_{i=0}^k\lambda_i\mu+\bar\beta}\\
    &= y_0 -\frac{\sum_{i=0}^{k} \lambda_{i}\mu\left( \sum_{t=0}^{i-1} \frac{\lambda_{t} \bar g_t}{\sum_{s=0}^{t}\lambda_{s}\mu + \bar \beta} \right) + \sum_{i=0}^{k}\lambda_i \bar g_i}{\sum_{i=0}^{k}\lambda_i\mu + \bar\beta}\\
    &= y_0 -\frac{\sum_{t=0}^{k-1} \lambda_t \bar g_t\frac{\sum_{i=t+1}^k \mu\lambda_i}{\sum_{i=0}^t \mu\lambda_{{\color{blue} i}} +\bar\beta} + \sum_{i=0}^{k}\lambda_i \bar g_i}{\sum_{i=0}^{k}\lambda_i\mu+\bar\beta}\\
    &= y_0 -\frac{\sum_{t=0}^{k-1} \lambda_t \bar g_t\left(\frac{\sum_{i=t+1}^k \mu\lambda_i}{\sum_{i=0}^t \mu\lambda_{{\color{blue} i}} +\bar\beta}+1\right)}{\sum_{i=0}^{k}\lambda_i\mu+\bar\beta} - \frac{\lambda_k \bar g_k}{\sum_{i=0}^{k}\lambda_i\mu + \bar\beta}\\
    &=  y_0 -\sum_{t=0}^{k-1} \frac{\lambda_t \bar g_t}{\sum_{i=0}^{t}\lambda_{{\color{blue} i}}\mu+\bar\beta} - \frac{\lambda_k \bar g_k}{\sum_{i=0}^{k}\lambda_i\mu + \bar\beta}\\
    &= x_k - \alpha_k\bar g_k
    \end{align*}
    where the second equality uses the inductive hypothesis for each $y_i-y_0 = x_i-x_0 = - \sum_{t=0}^{i-1} \alpha_t\bar g_t$, the third exchanges summands, and the remainder combines and simplifies terms. If $x_k$ is infeasible, {\color{blue} it is} immediate that $x_{k+1}=x_k-\alpha_k\bar g_k=y_{k+1}$. If $x_k$ is feasible, the above equality ensures
    $$ \frac{1}{\alpha_k}(x_k - \alpha_k g_0(x_k;\xi_k) -y_{k+1}) = n_{k+1} \in\partial r(y_{k+1}) \ . $$
    Noting by~\eqref{eq:proximal-equivalence}, $x_{k+1}$ is the unique solution to $\frac{1}{\alpha_k}(x_k-\alpha_kg_0(x_k;\xi_k) -z)\in\partial r(z)$, we must have $x_{k+1} {\color{blue} = x_k - \alpha_k(g_k + n_{k+1})} = y_{k+1}$ and $n_{k+1} =  \frac{1}{\alpha_k}(x_k - \alpha_k g_0(x_k;\xi_k) -x_{k+1})$ as required.
\end{proof}
\begin{remark}\label{rmk:weights}
Theorem~\ref{thm:super-equivalence} suffices to give a dual description for any sequence of primal stepsizes with $\alpha_0\in (0,1/\mu]$ and $\alpha_k\in (0,1/\mu)$ thereafter: one can select any $\lambda_0$ and $\bar\beta$ satisfying $\alpha_0 = \lambda_0/(\mu\lambda_0+\bar\beta)$ and then the corresponding sequence of dual weights is given by the recurrence
\begin{equation}\label{eq:weights}
        \lambda_{k+1} = \dfrac{\alpha_{k+1}}{1-\mu\alpha_{k+1}}\frac{\lambda_k}{\alpha_k} \qquad \qquad \left(\implies \lambda_{T} = \dfrac{\alpha_{T}}{\Pi_{k=1}^T(1-\mu\alpha_{k})}\frac{\lambda_0}{\alpha_0}\right)\ .
\end{equation}
One can inductively verify this sequence has $\alpha_k = \frac{\lambda_k}{\mu \sum_{i=0}^k \lambda_i + \bar\beta}$ as
$$ \frac{\lambda_{k+1}}{\mu \sum_{i=0}^{k+1} \lambda_i + \bar\beta} = \frac{\lambda_{k+1}}{\mu \sum_{i=0}^{k} \lambda_i + \bar\beta + \mu\lambda_{k+1}} = \frac{\lambda_{k+1}}{\frac{\lambda_k}{\alpha_k} + \mu\lambda_{k+1}} = \frac{\frac{\alpha_{k+1}}{1-\mu\alpha_{k+1}}}{1 + \frac{\mu\alpha_{k+1}}{1-\mu\alpha_{k+1}}} = \alpha_{k+1} \ . $$
Hence, provided dual weights $\lambda_k$, one can easily construct $\alpha_k$ as stated in the theorem, and conversely, given stepsizes $\alpha_k$, one can easily construct corresponding weights $\lambda_k$. For example, fixing $\lambda_0=1$, $\bar\beta=0$ multipliers for several common stepsizes are
    \begin{align*}
        \alpha_k &= \frac{1}{\mu(k+1)} \implies \lambda_k = 1 \ ,\\
        \alpha_k &= \frac{2}{\mu(k+2)} \implies \lambda_k = k+1 \ ,\\
        \alpha_k &= \frac{1}{\mu\sqrt{k+1}} \implies \lambda_k = \frac{1/\sqrt{k+1}}{\Pi_{i=1}^k (1-1/\sqrt{i+1})} \approx \exp(\sqrt{k})/\sqrt{k} \ .
    \end{align*}
\end{remark}
\begin{remark}
    Nesterov~\cite{Nesterov2005PrimaldualSM} noted that in non-strongly convex settings ($\mu=0$), decreasing stepsizes $\alpha_k$ corresponds to placing decreasing weight on new subgradient lower bounds $\lambda_k = \alpha_k \bar\beta$. This runs counter to the intuition that the newest models ought to be most relevant. Rather surprisingly, our Theorem~\ref{thm:super-equivalence} shows this fault does not extend to the strongly convex settings ($\mu>0$). As seen above, the decreasing stepsize selection of $\alpha_k = 2/\mu(k+2)$ corresponds to increasing dual weights $\lambda_k=k+1$.
\end{remark}

\subsection{Statement of Primal-Dual Convergence Guarantees}
For ease of presenting our convergence theory, we fix $\bar\beta=0$. This parameter's primary purpose in Nesterov's development of dual averaging~\cite{Nesterov2005PrimaldualSM} was to make the model subproblem strongly convex. Strongly convex problems $\mu>0$, as considered here, have no such need. Following Remark~\ref{rmk:weights}, fixing $\bar\beta=0$ only restricts the first stepsize as any sequence $\alpha_0=1/\mu$ and $\alpha_k\in (0,1/\mu)$ can still be dually described.

We prove a uniform convergence guarantee in terms of the primal gap
\begin{equation} \label{eq:primalgap-definition}
    \mathtt{primal\mbox{-}gap}_T := \frac{\sum_{k< T: s(x_k)=0} \lambda_k h_{x_\mathtt{OPT}}(x_k)}{\sum_{k< T: s(x_k)=0} \lambda_k } - p_\star \ ,
\end{equation}
which utilizes a combination of the feasible objective values seen, the dual gap
\begin{equation} \label{eq:dualgap-definition}
    \mathtt{dual\mbox{-}gap}_T := p_\star - \inf \frac{M^{(T-1)}}{\sum_{k< T: s(x_k)=0} \lambda_k }\ ,
\end{equation}
which utilizes a combination of the subgradient lower bounds seen, and the distance to optimal.
For the primal and dual gaps to be well-defined, at least one feasible iterate must have been seen (i.e., $\sum_{k< T: s(x_k)=0} \lambda_k >0$). Our assumptions facilitate a bound on how long it takes for this to occur. We show the expected fraction of the dual weight that occurs on iterations with a feasible iterate (i.e., $s(x_k)=0$) is bounded below.

\begin{proposition}\label{prop:slater-ratio}
    Under Assumptions A-C, for any primal stepsizes $\alpha_k>0$ and dual weights $\lambda_k>0$ satisfying~\eqref{eq:alpha-lambda} with $\bar\beta=0$, the stochastic switching proximal subgradient method~\eqref{eq:super-primal-method} has
    $$ \mathbb{E}_\xi\left[\frac{\sum_{k< T: s(x_k)=0} \lambda_k}{\sum_{k=0}^{T-1} \lambda_k}\right] \geq \frac{ \tau_\mathtt{SL} }{2(h_{x_\mathtt{SL}}(x_\mathtt{SL}) - \inf h_{x_\mathtt{SL}}) + \tau_\mathtt{SL}} \left(1 - \frac{L_0^2\sum_{k=0}^{T-1} \lambda_k\alpha_k + C_0}{{\color{blue} \tau_\mathtt{SL}}\sum_{k=0}^{T-1} \lambda_k}\right) $$
    where $\tau_\mathtt{SL} = 0-\max_{s=1\dots m} f_s(x_\mathtt{SL}) >0$.
\end{proposition}
A proof of this is given in Subsection~\ref{subsec:proof-slater-ratio}. In deterministic settings, a feasible iterate must then have been reached once this bound is positive. From this, we see that any stepsize selection with $\sum \lambda_k\alpha_k/\sum \lambda_k \rightarrow 0$ will asymptotically have at least $\frac{ \tau_\mathtt{SL} }{2(h_{x_\mathtt{SL}}(x_\mathtt{SL}) - \inf h_{x_\mathtt{SL}}) + \tau_\mathtt{SL}}>0$ fraction of the dual weight on iterations with $x_k$ feasible.
Once a feasible iterate occurs, our primal and dual convergence measures are well-defined. Alternatively, one could assume the initialization $x_0$ is feasible to ensure these quantities are well-defined. In either case, Subsection~\ref{subsec:proof-main-result} proves the following primal-dual convergence guarantee as our main result.
\begin{theorem}\label{thm:super-rate}
    Under Assumptions A-C, for any primal stepsizes $\alpha_k>0$ and dual weights $\lambda_k>0$ satisfying~\eqref{eq:alpha-lambda} with $\bar\beta=0$, the stochastic switching proximal subgradient method~\eqref{eq:super-primal-method} has
    \begin{align*}
        &\mathbb{E}_\xi \left[\left(\frac{\sum_{k< T: s(x_k)=0} \lambda_k}{\sum_{k=0}^{T-1} \lambda_k}\right)\left(\mathtt{primal\mbox{-}gap}_T + \mathtt{dual\mbox{-}gap}_T\right) + \frac{\mu}{2}\|x_{T}-x_\mathtt{OPT}\|_2^2\right] \\
        &\qquad \qquad \leq \frac{ L_0^2\sum_{k=0}^{T-1}\lambda_k\alpha_k + C_0}{\sum_{k=0}^{T-1}\lambda_k} \ .
    \end{align*}
\end{theorem}

\begin{remark}
Theorem~\ref{thm:super-rate} recovers the primal convergence rate of~\cite[Section 3.2]{Lacoste-Julien2012} with $\alpha_k=2/(\mu(k+2))$ and extends it to be a primal-dual guarantee covering proximal, switching, and non-Lipschitz settings. Theorem~\ref{thm:super-equivalence} shows this stepsize corresponds to dual averaging with weights $\lambda_k=k+1$ {\color{blue}and $\bar\beta=0$}. When $m=0$, Theorem~\ref{thm:super-rate} ensures
$$ \mathbb{E}_\xi\left[\mathtt{primal\mbox{-}gap}_T + \mathtt{dual\mbox{-}gap}_T + \frac{\mu}{2}\|x_{T}-x_\mathtt{OPT}\|_2^2\right] \leq \frac{4L_0^2}{\mu (T+1)} + \frac{2C_0}{T(T+1)} $$
using that $\sum_{k=0}^{T-1} \lambda_k = T(T+1)/2$ and $\sum_{k=0}^{T-1} \lambda_k\alpha_k = 2(T - \sum_{k=2}^{T+1} 1/k)/\mu\leq 2T/\mu$.
When $m>0$ and the subgradient oracle is deterministic, applying Proposition~\ref{prop:slater-ratio} gives a rate worse by only a factor depending on the Slater point of
\begin{align*}
    &\mathtt{primal\mbox{-}gap}_T + \mathtt{dual\mbox{-}gap}_T + \frac{\mu}{2}\|x_{T}-x_\mathtt{OPT}\|_2^2\\
    &\qquad\qquad\leq \frac{2(h_{x_\mathtt{SL}}(x_\mathtt{SL}) - \inf h_{x_\mathtt{SL}}) + \tau_\mathtt{SL}}{{\color{blue}\tau_\mathtt{SL} - \frac{4L_0^2}{\mu (T+1)} - \frac{2C_0}{T(T+1)}}} \left( \frac{4L_0^2}{\mu (T+1)} + \frac{2C_0}{T(T+1)}\right) \ . 
\end{align*}
Here, the role of $C_0$, defined in~\eqref{eq:C0}, bounding the effect of the non-Lipschitz constant $L_1$ becomes clear. The only role $L_1$ plays in our rate is via $T_0$, which in turn $C_0$ may be exponentially large in (see Proposition~\ref{prop:exp-bound}). If $C_0$ is small, then convergence will be dominated by the classic $O(L_0^2/\mu T)$ term as the dependence on $C_0$ shrinks at a fast $O(1/T^2)$ rate.
\end{remark}

\begin{remark}
Theorem~\ref{thm:super-rate} further recovers and extends the classic linear convergence of proximal gradient descent for smooth, strongly convex optimization. Assume $m=0$ and $f_0$ is $\beta$-smooth with $g_0(x;\xi) = \nabla f_0(x)$. Then~\eqref{eq:NonLipschitz-condition} holds with $L_0=0$ and $L_1=2\beta>\mu$.\footnote{This can be verified by noting $h_{x_\mathtt{OPT}}$ is $\beta$-smooth. Then the standard descent lemma ensures $$h_{x_\mathtt{OPT}}(x_\mathtt{OPT})\leq h_{x_\mathtt{OPT}}(x_{k} - (\nabla f_0(x_k)+n_{x_\mathtt{OPT}})/\beta) \leq h_{x_\mathtt{OPT}}(x_k) - \frac{1}{2\beta}\|\nabla f_0(x_k)+n_{x_\mathtt{OPT}}\|^2_2.$$} Consider the stepsize selection with $\alpha_0 = 1/\mu$ and $\alpha_k = 1/L_1$ constant thereafter, which corresponds to dual weights $\lambda_0=1$ and $\lambda_k = \frac{\mu}{L_1}(1 - \mu/L_1)^{-k}$ {\color{blue}, for $k>0$ }. This choice has $T_0=0$ and $C_0 = (\frac{L_1}{\mu}-1)\delta_0(x_\mathtt{OPT})$, giving the following linear convergence
$$ \mathtt{primal\mbox{-}gap}_T + \mathtt{dual\mbox{-}gap}_T + \frac{\mu}{2}\|x_{T}-x_\mathtt{OPT}\|_2^2\leq \frac{L_1}{\mu}\delta_0(x_\mathtt{OPT})\left(1 - \frac{\mu}{L_1}\right)^{T}, $$
using that $\sum_{k=0}^{T-1} \lambda_k = (1-\mu/L_1)^{-(T-1)} $.
\end{remark}

\begin{remark}
Theorem~\ref{thm:super-rate} provides new non-Lipschitz conditions for limiting primal-dual guarantees: Under Assumptions A-C and given a deterministic subgradient oracle,
$$\lim_{T\rightarrow \infty} \max\left\{ \mathtt{primal\mbox{-}gap}_T,\ \mathtt{dual\mbox{-}gap}_T,\|x_{T}-x_\mathtt{OPT}\|_2^2 \right\} = 0 $$
if $C_0$ is finite and $\sum_{k=0}^T \lambda_k\alpha_k/\sum_{k=0}^T \lambda_k \rightarrow 0$. Note this implies $\sum_{k=0}^T \lambda_k\rightarrow\infty$ since $\alpha_k,\lambda_k>0$. The classic conditions needed for limiting primal convergence under Lipschitz continuity are $\alpha_k\rightarrow 0$, which implies $C_0$ is finite, and $\sum^{T}_{k=0} \alpha_k^2/\sum^{T}_{k=0} \alpha_k \rightarrow 0$, which differs slightly from our theory when $\alpha_k=\lambda_k/\mu\sum_{i=0}^k\lambda_i$ as
$$ \text{classically, one needs } \frac{\sum_{k=0}^T \frac{\lambda_k\alpha_k}{\sum_{i=0}^k \lambda_i}}{\sum_{k=0}^T \frac{\lambda_k}{\sum_{i=0}^k \lambda_i}}\rightarrow 0 \quad \text{whereas we require} \quad \frac{\sum_{k=0}^T \lambda_k\alpha_k}{\sum_{k=0}^T \lambda_k}\rightarrow 0 \ . $$
\end{remark}

\begin{remark}
One can select stepsizes to minimize our rate. Given $C_0=0$ and the first $T$ stepsizes $\alpha_0,\dots \alpha_{T-1}$ and corresponding weights $\lambda_0,\dots\lambda_{T-1}$, one can select $\alpha_T$ and $\lambda_T$ to minimize our convergence upper bound~\eqref{eq:easy-bound} after one more step by setting\footnote{The formula for the optimal $\lambda_T$ in~\eqref{eq:optimal-lambda} can be verified as the unique solution to $\frac{d}{d\lambda_T}\left(\frac{\sum_{k=0}^T \lambda_k\alpha_k}{\sum_{k=0}^T\lambda_k}\right)=0$.}
\begin{equation}\label{eq:optimal-lambda}
    \alpha_{T} = \frac{\lambda_T}{\mu\sum_{k=0}^T \lambda_k} \quad \text{ and } \quad \lambda_{T} = \frac{\sum_{k=0}^{T-1} \lambda_k \times \sum_{k=0}^{T-1}\lambda_k\alpha_k}{\sum_{k=0}^{T-1} \lambda_k(2/\mu - \alpha_k)} \ .
\end{equation}
Given $\alpha_0=1/\mu$ and $\lambda_0=1$, the numerically optimized parameters and rate are below.

\begin{tabular}{ | r | c c c c c c c c c|}
\hline
$k$ & 0 & 1 & 2 & 3 & 4 & 5 & 6 & 7 & 8  \\ \hline
$\lambda_k$ & 1 & 1 & 1.2 & 1.4022 & 1.6025 & 1.8005 & 1.9966 & 2.1910 & 2.3841\\  
$\alpha_k$ & $\frac{1}{\mu}$ & $\frac{1}{2\mu}$ & $\frac{1}{2.6666\mu}$ & $\frac{1}{3.2820\mu}$ & $\frac{1}{3.8719\mu}$ & $\frac{1}{4.4460\mu}$ & $\frac{1}{5.0094\mu}$ & $\frac{1}{5.5648\mu}$  & $\frac{1}{6.1142\mu}$\\
Rate~\eqref{eq:easy-bound} & $\frac{L_0^2}{\mu}$ & $\frac{L_0^2}{1.3333\mu}$ & $\frac{L_0^2}{1.6410\mu}$ & $\frac{L_0^2}{1.9359\mu}$ & $\frac{L_0^2}{2.2230\mu}$ & $\frac{L_0^2}{2.5047\mu}$ & $\frac{L_0^2}{2.7824\mu}$ & $\frac{L_0^2}{3.0571\mu}$ & $\frac{L_0^2}{3.3293\mu}$ \\
\hline
\end{tabular}

\noindent For comparison, this offers small gains over the ``typical'' stepsize $\alpha_k = 2/\mu(k+2)$, shown below. Numerics showing some small gains actually occur are in Section~\ref{sec:numerics}.

\begin{tabular}{ | r | c c c c c c c c c|}
\hline
$k$ & 0 & 1 & 2 & 3 & 4 & 5 & 6 & 7 & 8  \\ \hline
$\lambda_k$ & 1 & 2 & 3 & 4 & 5 & 6 & 7 & 8 & 9 \\  
$\alpha_k$ & $\frac{1}{\mu}$ & $\frac{1}{1.5\mu}$ & $\frac{1}{2\mu}$ & $\frac{1}{2.5\mu}$ & $\frac{1}{3\mu}$ & $\frac{1}{3.5\mu}$ & $\frac{1}{4\mu}$ & $\frac{1}{4.5\mu}$  & $\frac{1}{5\mu}$\\
Rate~\eqref{eq:easy-bound} & $\frac{L_0^2}{\mu}$ & $\frac{L_0^2}{1.2857\mu}$ & $\frac{L_0^2}{1.5652\mu}$ & $\frac{L_0^2}{1.8404\mu}$ & $\frac{L_0^2}{2.1126\mu}$ & $\frac{L_0^2}{2.3824\mu}$ & $\frac{L_0^2}{2.6504\mu}$ & $\frac{L_0^2}{2.9168\mu}$ & $\frac{L_0^2}{3.1819\mu}$ \\
\hline
\end{tabular}
\end{remark}

\begin{remark}
    For deterministic settings {\color{blue} without regularization (i.e., supposing $g_s(x;\xi)$ is independent of $\xi$ and $r=0$ which allow expectations and $n_{y}$ to be omitted from Theorem~\ref{thm:super-rate} respectively)} where $\mu$ and an upper bound $G^2 \geq L_0^2$ are known, one can also utilize our theory to adapt stepsizes to avoid any early exponential divergences. Recall such divergences are quantified by $C_0$ as discussed at the end of Section~\ref{sec:prelim}. If one selects decreasing stepsizes with $\alpha_0=1/\mu$ and $\alpha_k \leq 1/L_1$ thereafter {\color{blue} and $\lambda_0=1$}, $T_0$ defined in~\eqref{eq:T0} {\color{blue} is at most} zero, and hence no bad divergence can occur as
    \begin{align}
        &\left(\frac{\sum_{k< T: s(x_k)=0} \lambda_k}{\sum_{k=0}^{T-1} \lambda_k}\right)\left(\mathtt{primal\mbox{-}gap}_T + \mathtt{dual\mbox{-}gap}_T\right) \nonumber\\
        &\qquad\qquad\leq \frac{ G^2\sum_{k=0}^{T-1}\lambda_k\alpha_k + (\frac{1}{\alpha_1\mu}-1)\frac{\|g_{s(x_0)}(x_0;\xi_0)\|^2}{2\mu}}{\sum_{k=0}^{T-1}\lambda_k}\label{eq:checkable}
    \end{align}
    where we bounded $L_0^2 \leq G^2$ and $C_0 \leq(\frac{L_1}{\mu} - 1)\delta_0(y) \leq (\frac{1}{\alpha_1\mu}-1)\frac{\|g_{s(x_0)}(x_0;\xi_0)\|^2}{2\mu}$\footnote{{\color{blue} The inequality $\delta_0(y) \leq \|g_{s(x_0)}(x_0;\xi_0)\|^2/2\mu$ follows from the strong convexity of $f_{s(x_0)}$ and $n_{y}=0$ as
    \begin{align*}
        \delta_{0}(y) &= f_{s(x_0)}(x_0) - f_{s(x_0)}(y)\\
        &\leq f_{s(x_0)}(x_0) - \inf_z f_{s(x_0)}(z)\\ &\leq f_{s(x_0)}(x_0) - \inf_z \left(f_{s(x_0)}(x_0) + \langle g_{s(x_0)}(x_0;\xi_0), z-x_0\rangle +\frac{\mu}{2}\|z-x_0\|^2_2\right)\\
        &=\|g_{s(x_0)}(x_0;\xi_0)\|^2/2\mu.
    \end{align*}}} for either $y\in\{x_{\mathtt{OPT}}, x_{\mathtt{SL}}\}$. Notice every quantity in~\eqref{eq:checkable} is computable! Hence, if one selected generic decreasing stepsizes $\alpha_k$, without knowing $L_1$ to ensure $\alpha_k\leq 1/L_1$ {\color{blue} for $k\geq 1$}, one can still check if convergence is occurring at the above rate. If~\eqref{eq:checkable} fails at some iteration, one can conclude $\alpha_1 > 1/L_1$. In this case, one could reasonably restart the method with reduced stepsizes, via an exponential backtracking. 
\end{remark}

\begin{remark}
    Without strong convexity, one cannot guarantee convergence of a duality gap since a linear $M^{(k)}$ leads the duality gap to always be $0$ or $\infty$. Our theory can still be applied by a standard trick: Instead of unconstrained minimization ($m=0$) of a convex function $f_0$, one could minimize the closely related strongly convex function
    $$\tilde f_0(x) = f_0(x) + \frac{\epsilon}{2D^2}\|x-x_0\|^2 \ . $$
    This perturbed problem has minimum value at most $p_\star +\epsilon\frac{\|x_\mathtt{OPT}-x_0\|^2_2}{2D^2}$, and so any $\epsilon$-minimizer of $\tilde f_0$ is an $(1+\frac{\|x_\mathtt{OPT}-x_0\|^2_2}{2D^2})\epsilon$-minimizer for the original problem.
    
    Note $\tilde f_0$ is $\epsilon/D^2$-strongly convex and since $m=0$, one can select $x_\mathtt{SL} = x_\mathtt{OPT}$. Moreover, if $f_0$ was $M$-Lipschitz continuous {\color{blue} with a deterministic subgradient oracle}, then as a sum of Lipschitz and smooth components, the perturbed objective $\tilde f_0$ satisfies~\eqref{eq:NonLipschitz-condition} with $L_0^2 = 6M^2$ by Lemma~\ref{lem:AssumptionC-verification}. As a result, Theorem~\ref{thm:super-rate} ensures applying the subgradient method to $\tilde f_0$ with stepsize $\alpha_k = 2D^2/\epsilon(k+2)$ has perturbed primal-dual gap converge at a rate $\frac{24M^2D^2}{\epsilon (T+1)} + O(1/T^2)$. Similar perturbed primal-dual guarantees using a novel proof method were given by~\cite{Gutman2022}.
\end{remark}

\subsection{Proof of Primal-Dual Convergence Guarantees}
Our theory relies on two symmetric inductive results, one inequality slightly generalizing the classic primal analysis and one novel inequality based on our dual perspective, in Lemmas~\ref{lem:switching-proximal-primal-induction} and~\ref{lem:switching-proximal-dual-induction}. From these, we prove the feasibility guarantee Proposition~\ref{prop:slater-ratio} and our main result Theorem~\ref{thm:super-rate}.

First, we show an inductive relationship on the (expected, unnormalized, squared) distance from the iterates $x_k$ to either $x_{\mathtt{OPT}}$ or $x_{\mathtt{SL}}$ defined as
\begin{equation} \label{eq:Rk-definition}
R_k(y) := \left(\frac{\mu}{2}\sum_{i=0}^{{k-1}} \lambda_i\right) \mathbb{E}_\xi\|x_k-y\|_2^2
\end{equation}
To simplify notations, throughout our analysis, we denote $g_k = g_{s(x_k)}(x_k;\xi_k)$ and $w_k = \mu\sum_{i=0}^{k} \lambda_i$ (with the convention that $w_{-1}=0$ as the given summation is empty).
\begin{lemma}\label{lem:switching-proximal-primal-induction}
    Under Assumptions A-C, the switching proximal subgradient method~\eqref{eq:super-primal-method} with $\alpha_k=\lambda_k/\mu\sum_{i=0}^k\lambda_i$ has for either $y\in\{x_{\mathtt{OPT}}, x_{\mathtt{SL}}\}$
    $$R_{k+1}(y) \leq R_k(y) - \frac{\lambda_k}{2}\left(\left(2-L_1\alpha_k\right)\mathbb{E}_\xi\delta_{k}(y) - L_0^2\alpha_k\right) \ . $$
\end{lemma}
\begin{proof}
    This proof follows a standard analysis technique, directly expanding the definition of $R_{k+1}(y)$. First, suppose $x_k$ is feasible. Then
    \begin{align*}
    R_{k+1}(y) &= \frac{w_k}{2}\mathbb{E}_\xi\|\mathrm{prox}_{\alpha_k,r}(x_k - \alpha_k g_k) - \mathrm{prox}_{\alpha_k,r}(y +\alpha_k n_{y})\|_2^2\\
    &\leq \frac{w_k}{2}\mathbb{E}_\xi\|x_k - \alpha_k(g_k+n_{y}) - y\|_2^2\\
    &= \frac{w_k}{2}\mathbb{E}_\xi\|x_k-y\|_2^2 - \lambda_k\mathbb{E}_\xi\langle g_k + n_{y}, x_k-y\rangle +\frac{\lambda_k\alpha_k}{2}\mathbb{E}_\xi\|g_k+n_{y}\|_2^2\\
    &\leq \frac{w_k}{2}\mathbb{E}_\xi\|x_k-y\|_2^2 - \lambda_k\mathbb{E}_\xi(h_y(x_k)-h_y(y)+\frac{\mu}{2}\|x_k - y\|_2^2) + \frac{\lambda_k\alpha_k}{2}\mathbb{E}_\xi\|g_k+n_{y}\|_2^2\\
    &= R_k(y) - \lambda_k\mathbb{E}_\xi\delta_k(y)+\frac{\lambda_k\alpha_k}{2}\mathbb{E}_\xi\|g_k+n_{y}\|_2^2
   \end{align*}
    where the first line uses that $\mathrm{prox}_{\alpha_k,r}(y +\alpha_k n_{y}) = y$, the second uses the nonexpansiveness of the proximal operator~\cite[Proposition 12.19]{RockWets98}, the third factors the norm squared and uses $\alpha_k = \lambda_k/w_k$, the fourth uses the strong convexity of $h_y$, {\color{blue} and the fifth follows from definitions as $R_k(y) = \frac{w_k - \lambda_k\mu}{2}\mathbb{E}_\xi\|x_k-y\|_2^2$ and $\delta_k(y)=h_y(x_k)-h_y(y)$}.
    Then, applying the bound~\eqref{eq:NonLipschitz-condition} gives the claim.
    Similarly, supposing $x_k$ is infeasible,
    \begin{align*}
    R_{k+1}(y) &= \frac{w_k}{2}\mathbb{E}_\xi\|x_k - \alpha_k g_k - y\|_2^2\\
    &= \frac{w_k}{2}\mathbb{E}_\xi\|x_k-y\|_2^2 - \lambda_k\mathbb{E}_\xi\langle g_k, x_k-y\rangle + \frac{\lambda_k\alpha_k}{2}\mathbb{E}_\xi\|g_k\|_2^2\\
    &\leq \frac{w_k}{2}\mathbb{E}_\xi\|x_k-y\|_2^2- \lambda_k\mathbb{E}_\xi(f_{s(x_k)}(x_k)-f_{s(x_k)}(y) +\frac{\mu}{2}\|x_k - y\|_2^2) + \frac{\lambda_k\alpha_k}{2}\mathbb{E}_\xi\|g_k\|_2^2\\
    &= R_k(y) - \lambda_k\mathbb{E}_\xi\delta_k(y)+\frac{\lambda_k\alpha_k}{2}\mathbb{E}_\xi\|g_k\|_2^2 
   \end{align*}
   using strong convexity of $f_{s(x_k)}$. Applying~\eqref{eq:NonLipschitz-condition} completes the proof.
\end{proof}
The dual portion of our convergence analysis relies on showing the same inductive relationship on the (expected, unnormalized) dual gap defined as
\begin{equation} \label{eq:Dk-definition}
    D_k := \mathbb{E}_\xi\left[\left(\sum_{i< k:s(x_i)=0} \lambda_i\right) p_\star - \inf M^{(k-1)}\right] \ .
\end{equation}
{\color{blue} Note this relates to the dual gap as $D_k = \mathbb{E}_\xi\left[\left(\sum_{i< k:s(x_i)=0} \lambda_i\right)\mathtt{dual\mbox{-}gap}_k\right]$.} % Dividing through by $ \sum_{i< k:s(x_i)=0} \lambda_i$, once nonzero, gives the dual gap.
\begin{lemma}\label{lem:switching-proximal-dual-induction}
    Under Assumptions A-C, the switching proximal subgradient method~\eqref{eq:super-primal-method} with $\alpha_k=\lambda_k/\mu\sum_{i=0}^k\lambda_i$ has
     $$ D_{k+1} \leq D_k - \frac{\lambda_k}{2}\begin{cases}
         \left(2-L_1\alpha_k\right)\mathbb{E}_\xi\delta_{k}(x_\mathtt{OPT}) - L_0^2\alpha_k & \text{if } x_k \text{ is feasible}\\
         {\color{blue}\left(2-L_1\alpha_k\right)\mathbb{E}_\xi\delta_{k}(x_\mathtt{OPT}) + 2\mathbb{E}_\xi f_{s(x_k)}(x_\mathtt{OPT}) - L_0^2\alpha_k} & \text{otherwise.}
     \end{cases} $$
\end{lemma}
\begin{proof}
Observe that one can rewrite $M^{(k)}(x) = Q_1(x) + Q_2(x)$ as the sum of two quadratics where $Q_1(x) = M^{(k-1)}(x)$ and $Q_2$ depends on whether $x_k$ is feasible. In the notation of Lemma~\ref{lem:quad-calc}, Lemmas~\ref{lem:Mk-minimization} {\color{blue} and~\ref{lem:standard-quadratic}} ensure $Q_1$ has $a_1 = \inf M^{(k-1)}$, $b_1 = \mu\sum_{i=0}^{k-1}\lambda_i$, and $z_1 = x_{k}${\color{blue}, since $\bar \beta = 0$ here}. To determine $Q_2$, first suppose $x_k$ is feasible. Then we have
$$Q_2(x) = \lambda_k \left(f_{0}(x_k) + \langle g_k, x-x_k \rangle +\frac{\mu}{2}\|x - x_k\|_2^2 + r(x_{k+1}) + \langle n_{k+1}, x-x_{k+1}\rangle\right) \ . $$
This quadratic can be written in the form $Q_2(z) = a_2 + \frac{b_2}{2}\|z-z_2\|_2^2$ with
\begin{align*}
    a_2 &= \lambda_k\left(f_0(x_k) + r(x_{k+1}) + \alpha_k\langle n_{k+1}, g_{k}+n_{k+1}\rangle - \frac{1}{2\mu}\|g_k + n_{k+1}\|_2^2 \right)\ , \\
    b_2 &= \mu\lambda_k \ , \qquad z_2 = x_k - \frac{g_k+n_{k+1}}{\mu} \ . 
\end{align*}
By Lemma~\ref{lem:quad-calc}, the expected minimum value of the updated model $\mathbb{E}_\xi \inf M^{(k)}$ is
\begin{align*}
    &\mathbb{E}_\xi \inf M^{(k-1)} + \lambda_k \mathbb{E}_\xi\Bigg(f_0(x_k) + r(x_{k+1}) + \alpha_k \langle n_{k+1}, g_{k}+n_{k+1}\rangle - \frac{1}{2\mu}\|g_k + n_{k+1}\|_2^2\\
    & \qquad\qquad\qquad\qquad\qquad + \frac{\sum_{i=0}^{k-1}\lambda_i }{2\mu \sum_{i=0}^k \lambda_i}\left\|g_{k}+n_{k+1}\right\|_2^2\Bigg) \ . 
\end{align*}
From this, we conclude the lower bound
\begin{align*}
    \mathbb{E}_\xi \inf M^{(k)} &\geq \mathbb{E}_\xi \inf M^{(k-1)} + \lambda_k \mathbb{E}_\xi\Bigg(f_0(x_k) + r(x_\mathtt{OPT}) + \langle n_{x_\mathtt{OPT}}, x_{k+1}-x_\mathtt{OPT}\rangle  \\
    & \quad + \alpha_k \langle n_{k+1}, g_{k}+n_{k+1}\rangle - \frac{1}{2\mu}\|g_k + n_{k+1}\|_2^2 + \frac{\sum_{i=0}^{k-1}\lambda_i }{2\mu \sum_{i=0}^k \lambda_i}\left\|g_{k}+n_{k+1}\right\|_2^2\Bigg)\\
    &= \mathbb{E}_\xi \inf M^{(k-1)} + \lambda_k \mathbb{E}_\xi\Bigg(\delta_k(x_\mathtt{OPT}) + p_\star + \langle n_{x_\mathtt{OPT}}, -\alpha_k(g_k + n_{k+1})\rangle  \\
    & \quad + \alpha_k \langle n_{k+1}, g_{k}+n_{k+1}\rangle - \frac{\alpha_k}{2}\|g_k + n_{k+1}\|_2^2\Bigg)\\
    &= \mathbb{E}_\xi\inf M^{(k-1)} + \lambda_k \left(\mathbb{E}_\xi\delta_k(x_\mathtt{OPT}) + p_\star - \frac{\alpha_k}{2}\mathbb{E}_\xi\|g_k+n_{x_\mathtt{OPT}}\|_2^2 { \color{blue} + \frac{\alpha_k}{2}\|n_{k+1}-n_{x_\mathtt{OPT}} \|_2^2}\right)\\
    & {\color{blue} \geq \mathbb{E}_\xi\inf M^{(k-1)} + \lambda_k \left(\mathbb{E}_\xi\delta_k(x_\mathtt{OPT}) + p_\star - \frac{\alpha_k}{2}\mathbb{E}_\xi\|g_k+n_{x_\mathtt{OPT}}\|_2^2\right)}
\end{align*}
where the first inequality lower bounds $r(x_{k+1})$ by $r(x_\mathtt{OPT}) + \langle n_{x_\mathtt{OPT}}, x_{k+1}-x_\mathtt{OPT}\rangle$, the first equality applies the definitions of $\delta_k(x_\mathtt{OPT})$ in~\eqref{eq:delta_k-definition}, $\alpha_k$ in~\eqref{eq:alpha-lambda} {\color{blue} with $\bar \beta = 0$, and that $x_{k+1}-x_k=-\alpha_k(g_k+n_{k+1})$}, and the {\color{blue} second} equality combines and simplifies terms.
In terms of $D_k = \mathbb{E}_\xi[\sum_{i< k:s(x_i)=0} \lambda_i p_\star - \inf M^{(k-1)}]$, this gives the following 
recurrence
$$ D_{k+1} \leq D_k - \lambda_k \mathbb{E}_\xi\delta_k(x_\mathtt{OPT}) + \frac{\lambda_k\alpha_k}{2}\mathbb{E}_\xi\|g_k+n_{x_\mathtt{OPT}}\|_2^2 \ . $$
Applying~\eqref{eq:NonLipschitz-condition} gives the claim in this case.
Now suppose $x_k$ is infeasible. Then, noting $Q_2$ minimizes at $x_k - \frac{g_{k}}{\mu}$, Lemma~\ref{lem:standard-quadratic} ensures
\begin{align*}
    Q_2(x) &= \lambda_k\left(f_{s(x_k)}(x_k) - \frac{1}{2\mu}\|g_k \|_2^2 + \frac{\mu}{2}\left\|x - \left(x_k - \frac{g_{k}}{\mu}\right)\right\|_2^2\right) \ .
\end{align*}
Then by Lemma~\ref{lem:quad-calc}, the expected minimum value of the updated model $\mathbb{E}_\xi \inf M^{(k)}$ is given by
\begin{align*}
    &\mathbb{E}_\xi\inf  M^{(k-1)} + \lambda_k\mathbb{E}_\xi\left(f_{s(x_k)}(x_k) - \frac{1}{2\mu}\|g_k \|_2^2\right) + \frac{\lambda_k\sum_{i=0}^{k-1}\lambda_i }{{\color{blue} 2}\mu \sum_{i=0}^k \lambda_i}\mathbb{E}_\xi\left\|g_{k}\right\|_2^2 \\
    &= \mathbb{E}_\xi\inf M^{(k-1)} + \lambda_k\mathbb{E}_\xi f_{s(x_k)}(x_k) - \frac{\lambda_k\alpha_k}{2}\mathbb{E}_\xi\|g_k\|_2^2 \ . 
\end{align*}
{\color{blue}
In terms of $D_k = \mathbb{E}_\xi[\sum_{i< k:s(x_i)=0} \lambda_i p_\star - \mathbb{E}_\xi\inf M^{(k-1)}]$, this gives the recurrence
$$ D_{k+1} = D_k - \lambda_k \mathbb{E}_\xi\left[\delta_{k}(x_\mathtt{OPT}) + f_{s(x_k)}(x_{\mathtt{OPT}})\right] + \frac{\lambda_k\alpha_k}{2} \mathbb{E}_\xi\|g_k\|_2^2 \ . $$
Bounding $\mathbb{E}_\xi\|g_k\|_2^2$ by~\eqref{eq:NonLipschitz-condition} gives the claim in this last case.
}
% Noting $f_{s(x_k)}(x_\mathtt{OPT})\leq 0$, we conclude the lower bound
% $$\mathbb{E}_\xi \inf M^{(k)} \geq \mathbb{E}_\xi\inf M^{(k-1)} + \lambda_k \left(\mathbb{E}_\xi\delta_k(x_\mathtt{OPT}) - \frac{\alpha_k}{2}\mathbb{E}_\xi\|g_k\|_2^2\right) \ .$$
\end{proof}

As a direct consequence of our primal inductive lemma, we can bound the rate that $\delta_k(y)$ grows in the first $T_0$ iterations as being at most exponential. From this, one can explicitly upper bound $C_0$ exponentially in $T_0$.
\begin{proposition}\label{prop:exp-bound}
    Under Assumptions A-C, the switching proximal subgradient method~\eqref{eq:super-primal-method} with $\alpha_k=\lambda_k/\mu\sum_{i=0}^k\lambda_i$ has for either $y\in\{x_{\mathtt{OPT}}, x_{\mathtt{SL}}\}$
    \begin{align*}
        |\delta_k(y)| & \leq L_1\|x_k - y\|_2^2 + \frac{L_0^2}{L_1} \ , \\
         \mathbb{E}_\xi \|x_{T}-y\|^2_2 & \leq \left(1 + \frac{\max\{2, L_1/\mu-2\}L_1}{\mu}\right)^T\left(\|x_{0}-y\|^2_2 + \frac{L_0^2}{{\color{blue}L_1^2}} + \frac{L_0^2}{\mu\max\{2, L_1/\mu-2\}{\color{blue}L_1}}\right)\ . 
    \end{align*}
\end{proposition}
\begin{proof}
    We first claim $x_k$ satisfies
    $$ -\frac{L_0^2}{L_1} \leq \delta_{k}(y) \leq \sqrt{L_0^2 + L_1\delta_k(y)}\ \|x_k - y\|_2 \ . $$
    The first inequality lower bounding $\delta_k(y)$ is immediate from~\eqref{eq:NonLipschitz-condition}.
    For the second inequality, note that if $x_k$ is feasible, convexity of $h_y$ ensures that $h_{y}(y) \geq h_{y}(x_k) + \langle \mathbb{E}_{\xi_k} (g_{s(x_k)}(x_k;\xi_k) + n_{y}), y-x_k\rangle$. If $x_k$ is infeasible, $f_{s(x_k)}(y) \geq f_{s(x_k)}(x_k) + \langle \mathbb{E}_{\xi_k} g_{s(x_k)}(x_k;\xi_k), y-x_k\rangle$. In either case, Cauchy-Schwarz and Assumption C give the second inequality.
    Observe that if $\delta_k(y)<0$ , the first inequality above ensures $|\delta_k(y)| \leq \frac{L_0^2}{L_1} \leq L_1\|x_k - y\|_2^2 + \frac{L_0^2}{L_1}$. Instead, if $\delta_k(y)\geq 0$, squaring the second inequality above ensures $\delta_k(y)^2 \leq (L_0^2 + L_1\delta_k(y))\|x_k-y\|^2_2$, which implies
    $$\delta_k(y) \leq \frac{L_1+\sqrt{L_1^2 + 4L_0^2/\|x_k-y\|_2^2}}{2}\|x_k-y\|_2^2\leq  \frac{L_0^2}{L_1} + L_1 \|x_k - y\|_2^2,$$
    where the second inequality uses concavity to bound {\color{blue} $\sqrt{a+b} \leq \sqrt{a} + \frac{b}{2\sqrt{a}}$, given $a>0$}, completing the proposition's first claim.
    To prove the proposition's second claim, note
    \begin{align*}
         \mathbb{E}_\xi \|x_{k+1}-y\|^2_2 &\leq \mathbb{E}_\xi \|x_{k}-y\|^2_2 - \frac{\lambda_k}{\mu\sum_{i=0}^{k} \lambda_i}((2-L_1\alpha_k)\mathbb{E}_\xi \delta_k(y) - L_0^2\alpha_k)\\
         &\leq \mathbb{E}_\xi \|x_{k}-y\|^2_2 + \alpha_k(|2-L_1\alpha_k||\mathbb{E}_\xi \delta_k(y)| + L_0^2\alpha_k)\\
         &\leq\left(1 + \alpha_k|2-L_1\alpha_k|L_1\right) \mathbb{E}_\xi \|x_{k}-y\|^2_2 + \alpha_k|2-L_1\alpha_k|\frac{L_0^2}{L_1} + L_0^2\alpha^2_k\\
         &\leq \left(1 + \frac{\max\{2, L_1/\mu-2\}L_1}{\mu}\right) \mathbb{E}_\xi \|x_{k}-y\|^2_2 + \frac{\max\{2, L_1/\mu-2\}}{\mu}\frac{L_0^2}{L_1} + \frac{L_0^2}{\mu^2},
    \end{align*}
    where the first inequality uses Lemma~\ref{lem:switching-proximal-primal-induction} divided by $(\frac{\mu}{2}\sum_{i=0}^k\lambda_i)$, the second inequality applies simple upper bounds, the third inequality uses our bound on $|\delta_k(y)|$, and the fourth uses that $0<\alpha_k\leq 1/\mu$ and its consequence $|2-L_1\alpha_k|\leq \max\{2, L_1/\mu-2\}$. From this, the proposition's second claim follows as such recurrences of the form $a_{k+1} \leq b\cdot a_k + c$ satisfy $a_T \leq b^Ta_0 + c\frac{b^T-1}{b-1} \leq b^T(a_0 + \frac{c}{b-1})${\color{blue} , where $b>1$ and $c\geq0$}.
\end{proof}

\subsubsection{Proof of Proposition~\ref{prop:slater-ratio}} \label{subsec:proof-slater-ratio}
Noting $ R_0(x_\mathtt{SL}) =0$ and $ R_{T}(x_\mathtt{SL})  \geq 0$, inductively applying Lemma~\ref{lem:switching-proximal-primal-induction} with $y=x_\mathtt{SL}$ shows
$$ \sum_{k=0}^{T-1}\frac{\lambda_k}{2}\left((2 - L_1\alpha_k)\mathbb{E}_\xi\delta_k(x_\mathtt{SL}) - L_0^2\alpha_k\right) \leq 0\ . $$
From this, we find that
\begin{align*}
   0 & \leq \sum_{k=0}^{T-1}\lambda_k(L_1\alpha_k - 2)\mathbb{E}_\xi\delta_k(x_\mathtt{SL}) + \sum_{k=0}^{T-1}L_0^2\lambda_k\alpha_k\\
    &=\mathbb{E}_\xi\left[\sum_{k=0}^{T-1}\lambda_k(L_1\alpha_k - 2)(\max\{\delta_k(x_\mathtt{SL}),0\} + \min\{\delta_k(x_\mathtt{SL}),0\})\right] + \sum_{k=0}^{T-1}L_0^2\lambda_k\alpha_k\\
    & \leq C_0 + \mathbb{E}_\xi\left[\sum_{k=0}^{T-1}-\lambda_k\max\{\delta_k(x_\mathtt{SL}),0\} +\sum_{k=0}^{T-1}\lambda_k(L_1\alpha_k - 2)\min\{\delta_k(x_\mathtt{SL}),0\}\right] + \sum_{k=0}^{T-1}L_0^2\lambda_k\alpha_k\\
    & \leq C_0 - \tau_\mathtt{SL}\mathbb{E}_\xi\left[\sum_{k< T: s(x_k)\neq0}\lambda_k\right] + 2(h_{x_\mathtt{SL}}(x_\mathtt{SL}) - \inf h_{x_\mathtt{SL}})\mathbb{E}_\xi\left[\sum_{k< T: s(x_k)=0}\lambda_k\right] + \sum_{k=0}^{T-1}L_0^2\lambda_k\alpha_k
\end{align*}
where the first inequality uses our inductive result, the second inequality uses the definition of $C_0$ in~\eqref{eq:C0} and that $\delta_k(x_\mathtt{OPT})\geq 0$, and the third inequality bounds the first two summations as follows: (i) the first sum's upper bound notes that if $s(x_k)\neq 0$, then $\delta_k(x_\mathtt{SL}) \geq \tau_\mathtt{SL}>0$ and (ii) the second sum's upper bound notes $L_1\alpha_k-2 \geq -2$ and if $s(x_k)=0$, then $\delta_k(x_\mathtt{SL}) \geq \inf h_{x_\mathtt{SL}} - h_{x_\mathtt{SL}}(x_\mathtt{SL})$, which may be negative.
Rearrangement gives the claim as
\begin{equation*}
    \mathbb{E}_\xi \left[\sum_{k< T: s(x_k)=0}\lambda_k\right]\geq \frac{ \tau_\mathtt{SL} \sum_{k=0}^{T-1} \lambda_k - L_0^2\sum_{k=0}^{T-1} \lambda_k\alpha_k - C_0}{2(h_{x_\mathtt{SL}}(x_\mathtt{SL}) - \inf h_{x_\mathtt{SL}}) + \tau_\mathtt{SL}} \ . \qedhere
\end{equation*}

\subsubsection{Proof of Theorem~\ref{thm:super-rate}} \label{subsec:proof-main-result}
Applying Lemma~\ref{lem:switching-proximal-primal-induction} with $y=x_\mathtt{OPT}$ from $k=0$ to $T-1$ yields
$$R_T(x_\mathtt{OPT}) + \sum_{k=0}^{T-1} \frac{\lambda_k}{2}\mathbb{E}_\xi\delta_k(x_\mathtt{OPT}) \leq R_0(x_\mathtt{OPT}) + \sum_{k=0}^{T-1}\frac{L_0^2\lambda_k\alpha_k}{2} + \sum_{k=0}^{T-1} \frac{\lambda_k}{2}\left(L_1\alpha_k - 1 \right)\mathbb{E}_\xi\delta_k(x_\mathtt{OPT})\ .$$
Similarly, applying Lemma~\ref{lem:switching-proximal-dual-induction} from $k=0$ to $T-1$ yields
\begin{align*}
    D_T &+ \sum_{k=0}^{T-1} \frac{\lambda_k}{2}(\mathbb{E}_\xi\delta_k(x_\mathtt{OPT})) + {\color{blue}\sum_{k<T:s(x_k)>0} \lambda_k\mathbb{E}_\xi f_{s(x_k)}(x_\mathtt{OPT})}\\
    &\leq D_0 + \sum_{k=0}^{T-1}\frac{L_0^2\lambda_k\alpha_k}{2} + \sum_{k=0}^{T-1} \frac{\lambda_k}{2}\left(L_1\alpha_k - 1 \right)\mathbb{E}_\xi\delta_k(x_\mathtt{OPT})\ .
\end{align*}
Noting $R_0(x_\mathtt{OPT})=D_0=0$ and the last summation in each bound is at most half our initial blow-up constant $\frac{1}{2}C_0$, the sum of these inequalities provides a bound of
$$ R_T(x_\mathtt{OPT}) + D_T + \sum_{k=0}^{T-1} \lambda_k\mathbb{E}_\xi\delta_k(x_\mathtt{OPT}) {\color{blue} + \sum_{k<T:s(x_k)>0} \lambda_k\mathbb{E}_\xi}f_{s(x_k)}(x_\mathtt{OPT}) \leq \sum_{k=0}^{T-1}L_0^2\lambda_k\alpha_k + C_0 \ . $$
Lower bounding each $\mathbb{E}_\xi[\delta_k(x_\mathtt{OPT}) {\color{blue}+  f_{s(x_k)}(x_\mathtt{OPT})}]$ with $s(x_k)\neq 0$ by zero {\color{blue} and dividing through by $\sum_{k=0}^{T-1} \lambda_k$ establishes that
$$ \frac{1}{\sum_{k=0}^{T-1} \lambda_k}\left(R_T(x_\mathtt{OPT}) + D_T + \sum_{k<T:s(x_k)=0} \lambda_k\mathbb{E}_\xi\delta_k(x_\mathtt{OPT})\right) \leq \frac{\sum_{k=0}^{T-1}L_0^2\lambda_k\alpha_k + C_0}{\sum_{k=0}^{T-1} \lambda_k} \ . $$
Substituting the definitions of $R_T$ and $D_T$ from~\eqref{eq:Rk-definition} and~\eqref{eq:Dk-definition}, the left-hand side above is equal to
$$ \mathbb{E}_\xi\left[\frac{\mu}{2}\|x_T-x_\mathtt{OPT}\|^2_2 + \frac{1}{{\sum_{k=0}^{T-1} \lambda_k}}\left(\sum_{k<T:s(x_k)=0} \lambda_k h_{x_\mathtt{OPT}}(x_k)- \inf M^{(T-1)}\right)\right]\ . $$
Recalling the primal and dual gaps definitions in~\eqref{eq:primalgap-definition} and~\eqref{eq:dualgap-definition} gives our main theorem.}

    \section{Numerical Experiments}\label{sec:numerics}
In this section, we numerically validate the accuracy of Theorem~\ref{thm:super-rate} in predicting actual observed performance. Our three main numerical experiments address the impact of varying $\lambda_k$, the quality of our new primal-dual stopping criteria, and the accuracy of our $T_0$ and $C_0$ constants at predicting initial divergences. All of our numerics are implemented in \texttt{Julia 1.8.5}\footnote{The source code is available at \url{https://github.com/AshleyLDL/Primal-Dual-Averaging-Coding}}.

We consider the following deterministic family of nonsmooth, non-Lipschitz, strongly convex minimization problems given $A, C \in \mathbb{R}^{m\times n}$ and $b, d \in \mathbb{R}^{m}$
\begin{equation}\label{eq:numerical-example}
    \min_{x\in\mathbb{R}^n} f_0(x)=\|A x - b\|_1 + \frac{1}{2}\|C x-d\|_2^2 \ . 
\end{equation}
Note $\|A x - b\|_1$ is $\|A^T\|_{\infty\rightarrow 2}$-Lipschitz\footnote{The Lipschitz constant for $\|A x - b\|_1$ follows from the chain rule as its subgradients are combinations of $A$'s rows with weights in $[-1,1]$, so the largest subgradient is $\max_{\|w\|_\infty\leq 1} \|A^Tw\|_2 = \|A^T\|_{\infty\rightarrow 2}$.}. However, computing this induced matrix norm is NP-hard~\cite{steinberg-hardness}, so we instead upper bound it by $\sum_{i=1}^m \|A_i\|$ where $A_i$ denotes $A$'s $i$th row. Further noting $\frac{1}{2}\|C x-d\|_2^2$ is $\lambda_{max}(C^TC)$-smooth, by Lemma~\ref{lem:AssumptionC-verification}, our Assumptions A-C hold with $L_0^2 = 8(\sum_{i=1}^m \|A_i\|)^2$, $L_1 = 4\lambda_{max}(C^TC)$ and $\mu = \lambda_{min}(C^TC)$. 
We generate problem instances fixing $m=n=100, x_0=0$ and randomly drawing $A, \tilde C, x_\mathtt{OPT}$ with i.i.d.~normal entries. To control $\mu$ and $L_1$, we set $C = I + \sigma\tilde C $ for various selections of $\sigma\geq 0$. When $\sigma=0$, we have $\mu=1$ and $L_1=4$. Initially as $\sigma$ increases, $\mu$ decreases while $L_1$ increases. To ensure $x_\mathtt{OPT}$ is a minimizer and $p_\star = 0$, we set $b = Ax_\mathtt{OPT}$, $d = Cx_\mathtt{OPT}$.

\subsection{Performance under Varied Stepsize Selections}
First, we aim to measure the quality of Theorem~\ref{thm:super-rate}'s bounds compared to actual convergence. We fix $\bar\beta=0$ and $\sigma=0$ and consider several polynomial selections of $\lambda_k$ and our proposed, optimized choice~\eqref{eq:optimal-lambda}. Figure~\ref{fig:observed values for different lambda} shows the upper bound from Theorem~\ref{thm:super-rate} in comparison to the observed convergence of the aggregate measure $\mathtt{primal\mbox{-}gap}_T + \mathtt{dual\mbox{-}gap}_T + \frac{\mu}{2}\|x_{T}-x_\mathtt{OPT}\|_2^2$ and each component separately. As expected, the optimized parameters~\eqref{eq:optimal-lambda} have the best theoretical bound and the best observed aggregate performance early on. Moreover, it remains one of the best methods throughout. Asymptotically, we see comparable convergence for all $\lambda_k\neq 1$. The primal convergence under uniform weights $\lambda_k=1$ was the slowest in line with our theory, which only guarantees a $O(\log(T)/T)$ rate. Uniform weights did yield the fastest convergence of the dual gap and distance to optimal, which our theory cannot explain.

\begin{figure}[t]
\centering
\includegraphics[width=0.8\textwidth]{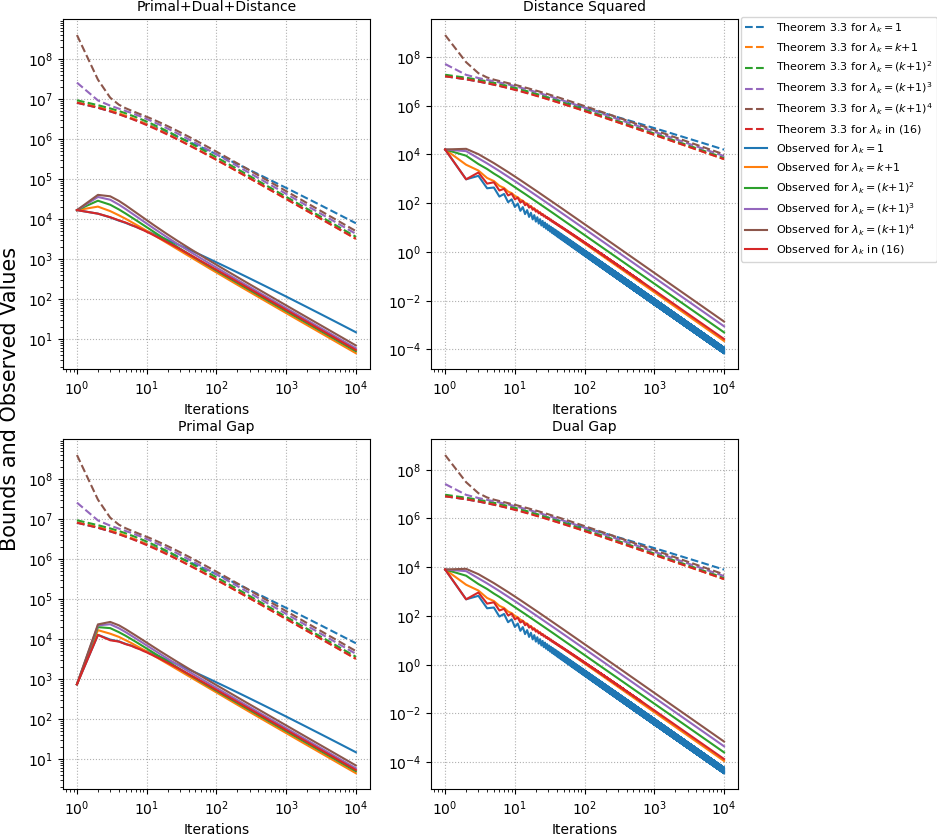}
\caption{Bounds and observed performance for different $\lambda_k$ with $\bar\beta=0$.}
\label{fig:observed values for different lambda}
\end{figure}

\subsection{High Accuracy of Primal-Dual Stopping Criteria} \label{subsubsec:numeric2}

One practical benefit of our dual characterizations of subgradient methods is the resulting computable dual lower bounds and hence stopping criteria, assuming $\mu$ is known. As a shorthand, denote the convergence of our dual lower bound on $p_\star$ by $d_k := p_\star - \inf M^{(k-1)}/\sum \lambda_i$. We denote the convergence of three natural upper bounds on $p_\star$ by the primal gap (averaging function values seen) $p_k := \sum \lambda_i f(x_i) / \sum \lambda_i-p_\star$, the function value at an averaged iterate $\bar p_k := f(\sum \lambda_i x_i/\sum \lambda_i)-p_\star$, and the function value at the latest iterate $\delta_k := f(x_k)-p_\star$. Combining these upper and lower bounds gives three natural stopping criteria to ensure an $\epsilon$-accurate solution is found: stopping once the gap between upper and lower bounds is less than $\epsilon$. Fixing $\sigma=0$ and $\epsilon=0.05$, Table~\ref{Stopping Criteria for Different Choices of lambda and beta} shows the number of iterations before these conditions were first reached.

Across every $\lambda_k\neq 1$ configuration, we see $\bar p_T$ and $d_T$ both converge relatively quickly. The stopping criteria $\bar p_t + d_t\leq \epsilon$ is consistently reached in at most $25\%$ more iterations than were required to reach $\bar p_t\leq \epsilon$. Hence, up to a small constant, this criterion matches the ideal time to stop. Note $\delta_t$ and $p_t$ both converged much slower than $\bar p_t$ and $d_t$. Correspondingly, the stopping criteria $\delta_t + d_t \leq \epsilon$ and $p_t+d_t\leq\epsilon$ are highly accurate, being reached in all of our experiments within one or two iterations of the first iteration with $\delta_t\leq \epsilon$ or $p_t\leq \epsilon$.

\begin{table}
\centering
\begin{tabular}{ |l||l|l|l|l|l|l|  }
 \hline
 \multicolumn{7}{|c|}{First $t$ satisfying the given stopping criteria} \\
 \hline
   Criteria  & $\lambda_k{=}1$ &$\lambda_k{=}k+1$&$\lambda_k{=}(k+1)^2$&$\lambda_k{=}(k+1)^3$&$\lambda_k{=}(k+1)^4$&$\lambda_k$ in~\eqref{eq:optimal-lambda}\\
 \hline
 $\bar p_t \leq \epsilon$ &1204821  & 1940 & 997 & 1331 & 1664 & 4122 \\
 $\bar p_t + d_t \leq \epsilon$   &1204821&2000 & 1223 & 1630 & 2038 & 4156\\
 \hline
 $\delta_t\leq \epsilon$&237426 &443222 & 664834 & 886445 & 1108056 & 533876\\
$\delta_t + d_t\leq \epsilon$ &237428&443223 & 664835 & 886446 & 1108058 & 533876\\
 \hline
 $p_t \leq \epsilon$ &4713468 &886456 & 997251 & 1181927 & 1385070 & 1067789\\
 $ p_t + d_t \leq \epsilon$ &4713468& 886456 & 997252 & 1181928 & 1385071 & 1067790\\
 \hline
 $d_t \leq \epsilon $ &263  &  470  & 705 & 941 &1176 & 509\\ \hline
\end{tabular}
\caption{Stopping times for different criteria and $\lambda_k$ with $\epsilon=0.05, \sigma=0$.}
\label{Stopping Criteria for Different Choices of lambda and beta}
\end{table}

\subsection{Accuracy of $C_0$ at Predicting Early Iterate Divergence} \label{subsubsec:numeric3}

\begin{table}
\centering
\begin{tabular}{ |p{1cm}|p{1.5cm}|p{1.5cm}|p{1.5cm}|p{1.6cm}|p{1.6cm}|p{1.7cm}|  }
 \hline
 \multicolumn{7}{|c|}{Conditioning of problems~\eqref{eq:numerical-example} as $\sigma$ varies} \\
 \hline
$\sigma$ & 0 & 0.0001 & 0.001 & 0.01 & 0.02 & 0.05\\
 $L_1/\mu$   & 4   & 4.022 & 4.224 & 6.911 & 12.107 &81.179  \\
 $T_0$   &   6  & 7 & 7 & 12 & 23 & 161\\
 $C_0$   &$1.472{\times}10^{5}$ & $1.497{\times}10^{5}$ & $1.735{\times}10^{5}$ & $6.985{\times}10^{5}$ & $3.770{\times}10^{6}$ & $2.663{\times}10^{23}$\\
 \hline
\end{tabular}
\caption{Effects of $\sigma$ on problem conditioning measured by $L_1/\mu$ and consequently the duration and amount of early divergences measured by $T_0$ and $C_0$ with $\alpha_k = 2/\mu(k+2)$.}
\label{Effects for Different Sigma}
\end{table}

Lastly, we consider settings where the initial iterates diverge rapidly, which our theory addresses via the inclusion of the constant $C_0$, defined in~\eqref{eq:C0}. Here, we have defined $C=\sigma\tilde C + I$, for a randomly Gaussian sampled $\tilde C$. As a result, the constants $\mu = \lambda_{min}(C^TC)$ and $L_1  = 4\lambda_{max}(C^TC)$ depend on $\sigma$. In Table~\ref{Effects for Different Sigma}, we show the effect $\sigma$ varying from $0$ to $0.05$, causing the condition number $L_1/\mu$ to grow moderately. As a result, we see $T_0$ grow linearly in $L_1/\mu$ and $C_0$ grows exponentially, exceeding $10^{23}$.

For such problems, our theory predicts the subgradient method with $\alpha_k=2/\mu(k+2)$ may diverge in the first $T_0$ iterations but should eventually converge at least a $O(1/T)$ rate. Figure~\ref{fig:observed values for different sigma} numerically confirms this prediction with every performance measure exponentially growing to at least $10^{16}$ as $\sigma$ grows and a decreasing trend beginning before iteration $T_0$. We see $\bar p_k$, $\delta_k$ and $R_k^2(x_\mathtt{OPT})$ rapidly converge after $T_0$, whereas $p_k$ and $d_k$ only decrease sublinearly. This slow convergence is likely due to $p_k$ and $d_k$ being defined as weighted averages, which must slowly dilute the effects of early ``bad'' iterations. Our theory predicts such exponential divergences can be avoided by ensuring $T_0$ (and hence $C_0$) are small. For example, setting $\alpha_0=1/\mu$ and then $\alpha_k = \min\{1/L_1, 2/\mu(k+2)\}$ thereafter rather than $\alpha_k=2/\mu(k+2)$ as above ensures $T_0=0$. Figure~\ref{fig:observed values for different sigma after modification on lambda} verifies this mitigates the previous diverging behavior.

\begin{figure}
\centering
\includegraphics[width=0.8\textwidth]{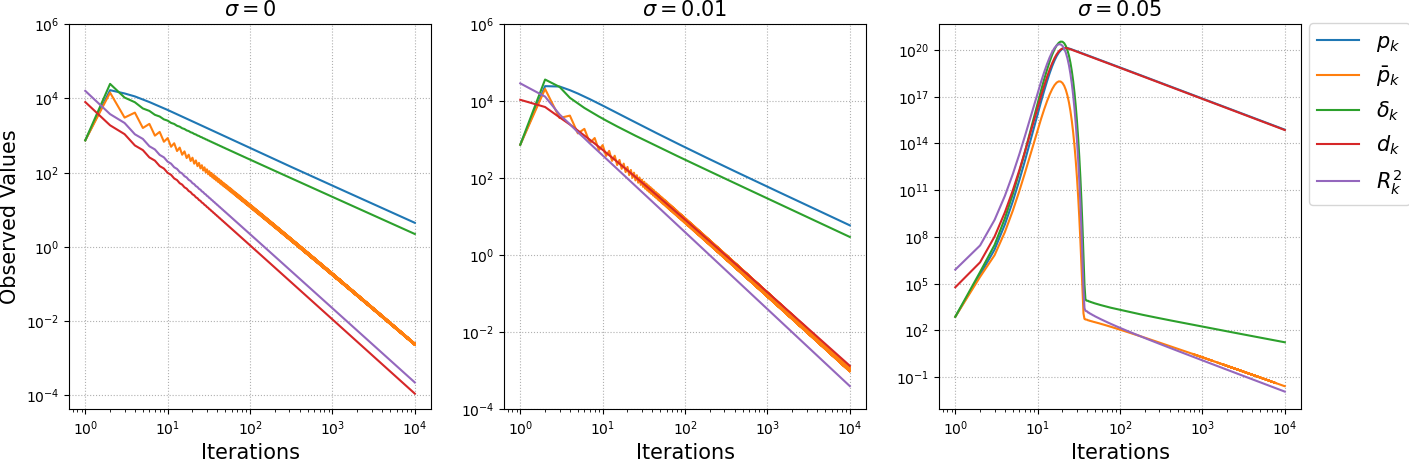}
\caption{Observed performance for various $\sigma$ with $\alpha_k = 2/\mu(k+2)$.}
\label{fig:observed values for different sigma}
\end{figure}

\begin{figure}
\centering
\includegraphics[width=0.8\textwidth]{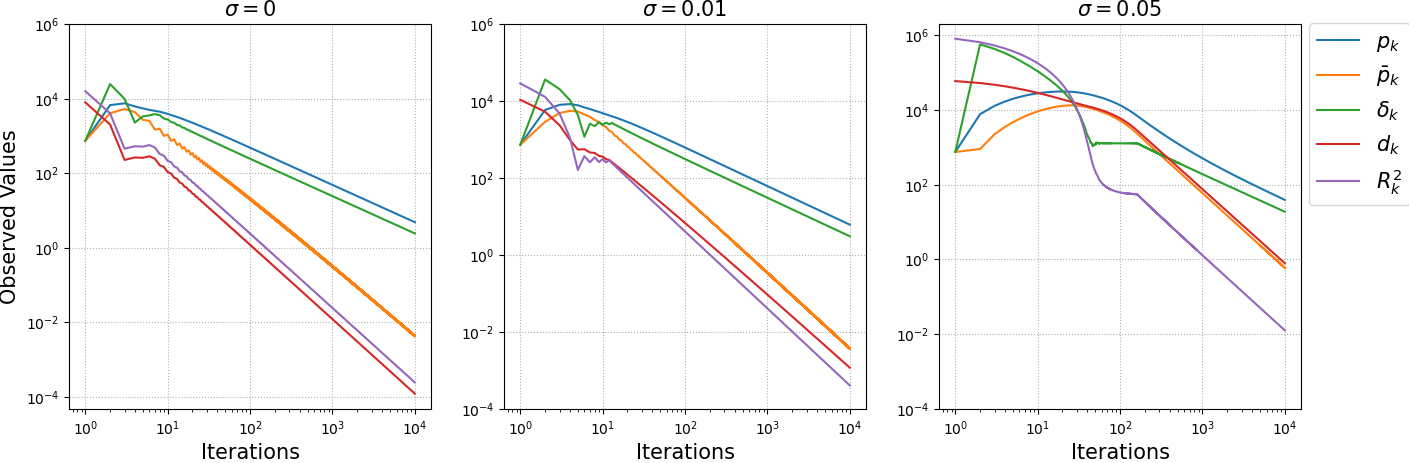}
\caption{Observed performance for various $\sigma$ with $\alpha_0=1/\mu$, $\alpha_k= \min\{1/L_1, 2/\mu(k+2)\}$, for $k>0$, with corresponding $\lambda_0 = 1$, $\lambda_k=\frac{\alpha_{k}}{1-\mu\alpha_k}\frac{\lambda_{k-1}}{\alpha_{k-1}}$ and well-controlled $T_0=0$.}
\label{fig:observed values for different sigma after modification on lambda}
\end{figure}

    {\color{blue} \paragraph{Acknowledgements.} This manuscript's presentation and clarity were greatly improved by the careful thoughts and advice of two anonymous referees to whom the authors are very grateful. 
    }
    % {\color{red}
    % \input{intro}
    % \input{results}
    % \input{proofs}
    % \input{proximal}
    % \input{switching}
    % }    
    {\small
    \bibliographystyle{unsrt}
    \bibliography{bibliography}

    \appendix\section{Deferred Proof of Lemma~\ref{lem:AssumptionC-verification}}
    Consider either $y\in\{x_{\mathtt{OPT}}, x_{\mathtt{SL}}\}$.
    Suppose first $x_k$ is feasible and let $\bar g_k =  \mathbb{E}_{\xi_k} g_0(x_k;\xi_k) + n_y\in\partial h_y(x_k)$.
    Fix any $g_y \in \partial h_y(y)$. Note by the sum rule of subdifferential calculus, both $\bar g_k - \nabla f^{(2)}_0(x_k) -n_y$ and $g_y - \nabla f^{(2)}_0(y) -n_y$ are subgradients of $f^{(1)}_0$ and hence both have norm bounded by $M$.
    Consider the $L$-smooth function
    $$\hat{h}_y(x) = f_0^{(2)}(x) + f^{(1)}_0(y) + \langle g_y - \nabla f_0^{(2)}(y), x-y\rangle {\color{blue} +r(y)} \ . $$
    Note since {\color{blue} $f_0^{(1)}$ is convex with} $g_y - \nabla f_0^{(2)}(y) - n_y \in\partial f^{(1)}_0(y)$, $h_y \geq \hat{h}_y$ and $h_y(y)=\hat{h}_y(y)$.
    Then one has
    \begin{align*}
        &\mathbb{E}_{\xi_k}\|g_0(x_k;\xi_k) + n_{y}\|_2^2\\
        &= \|\bar g_k\|_2^2 + \mathbb{E}_{\xi_k}\|g_0(x_k;\xi_k) {\color{blue} + n_y} - \bar g_k\|_2^2 \\ 
        & \leq 3\|\nabla \hat{h}_y(x_k)\|_2^2 + 3\|\bar g_k - \nabla f^{(2)}_0(x_k) -n_y\|^2 + 3\|g_y - \nabla f^{(2)}_0(y) -n_y\|_2^2 +\sigma^2 \\
        &\leq 6L(\hat{h}_y(x_k) - \inf \hat{h}_y) + 6M^2 + \sigma^2 \\
        &\leq 6L\delta_k(y) + 6L(h_y(y) - \inf \hat{h}_y) + 6M^2 + \sigma^2 
    \end{align*}
    where the first inequality bounds $\|a+b+c\|^2_2$ by $3\|a\|^2_2+3\|b\|^2_2+3\|c\|^2_2$ and uses the assumed variance bound, the second inequality uses smoothness to bound the first term as $\hat{h}_y(x_k) - \frac{1}{2L}\|\nabla \hat{h}_y(x_k)\|^2 \geq \inf \hat{h}_y$ and the $M$-Lipschitzness of $f^{(1)}_0$ to bound the second and third terms, and the final inequality adds and subtracts ${\color{blue}6L}h_y(y)$ and upper bounds $\hat{h}_y(x_k)$ by $h_y(x_k)$.
    
    Similarly, now suppose $x_k$ is infeasible and let $\bar g_k =  \mathbb{E}_{\xi_k} g_{s(x_k)}(x_k;\xi_k) \in\partial f_{s(x_k)}(x_k)$.
    Fix any $g_y \in \partial f_{s(x_k)}(y)$. Note by the sum rule of subdifferential calculus, both $\bar g_k - \nabla f^{(2)}_{s(x_k)}(x_k)$ and $g_y - \nabla f^{(2)}_{s(x_k)}(y)$ are subgradients of $f^{(1)}_{s(x_k)}$ and hence both have norm bounded by $M$.
    Consider the $L$-smooth function
    $$\hat{f}_{s(x_k)}(x) = f_{s(x_k)}^{(2)}(x) + f^{(1)}_{s(x_k)}(y) + \langle g_y - \nabla f_{s(x_k)}^{(2)}(y), x-y\rangle \ . $$
    Note since $g_y - \nabla f_{s(x_k)}^{(2)}(y) \in\partial f^{(1)}_{s(x_k)}(y)$, $f_{s(x_k)} \geq \hat{f}_{s(x_k)}$ and $f_{s(x_k)}(y)=\hat{f}_{s(x_k)}(y)$.
    Then, identical reasoning to that above gives
    \begin{align*}
        &\mathbb{E}_{\xi_k}\|g_0(x_k;\xi_k)\|_2^2\\
        &= \|\bar g_k\|_2^2 + \mathbb{E}_{\xi_k}\|g_{s(x_k)}(x_k;\xi_k) - \bar g_k\|_2^2 \\ 
        & \leq 3\|\nabla \hat{f}_{s(x_k)}(x_k)\|_2^2 + 3\|\bar g_k - \nabla f^{(2)}_{s(x_k)}(x_k)\|^2 + 3\|g_y - \nabla f^{(2)}_{s(x_k)}(y)\|_2^2 +\sigma^2 \\
        &\leq 6L(\hat{f}_{s(x_k)}(x_k) - \inf \hat{f}_{s(x_k)}) + 6M^2 + \sigma^2 \\
        &\leq 6L\delta_k(y) + 6L(f_{s(x_k)}(y) - \inf \hat{f}_{s(x_k)}) + 6M^2 + \sigma^2 \ .
    \end{align*}
    Hence, Assumption C holds with
    \begin{align*}
        L_0^2 &= 6M^2 + \sigma^2 + 6L\max_{y\in\{x_{\mathtt{OPT}}, x_{\mathtt{SL}}\}}\max_{s=1\dots m}\left\{h_y(y) - \inf \hat{h}_y, f_{s}(y) - \inf \hat{f}_{s}\right\}\\
        L_1 &=6L\ .
    \end{align*}
    }
\end{document}